\documentclass[12pt,oneside,reqno]{amsart}

\usepackage{amsmath,amssymb, amsthm,mathtools}
\usepackage[margin=1.5in]{geometry}
\usepackage[parfill]{parskip}
\usepackage{enumitem}
\usepackage{tikz}
\usepackage{tikz-cd} 
\usepackage{dsfont}
\usepackage[caption=false]{subfig}
\usepackage{mathrsfs}
\usepackage{bbm}
\usepackage{dirtytalk}
\usepackage{bm}
\usepackage{microtype}

\DeclarePairedDelimiter{\floor}{\lfloor}{\rfloor}
\providecommand{\N}{\mathbb{N}}
\providecommand{\F}{\mathbb{F}}
\providecommand{\R}{\mathbb{R}}
\providecommand{\C}{\mathbb{C}}
\providecommand{\T}{\mathbb{T}}
\providecommand{\Z}{\mathbb{Z}}

\renewcommand{\vec}[1]{\boldsymbol{#1}}

\providecommand{\leqst}{\leq_{\mathrm{st}}}

\newcommand{\paren}[1]{\left( #1 \right)}
\newcommand{\brac}[1]{\left[ #1 \right]}

\newcommand{\abs}[1]{\left\vert#1\right\vert}

\newcommand{\set}[1]{\left\{#1\right\}}

\DeclareMathOperator{\rank}{rank}

\newtheorem{Theorem}{Theorem}
\newtheorem{Lemma}[Theorem]{Lemma}
\newtheorem{Definition}[Theorem]{Definition}
\newtheorem{Proposition}[Theorem]{Proposition}

\newtheorem{Corollary}[Theorem]{Corollary}
\newtheorem{Conjecture}[Theorem]{Conjecture}

\newcounter{cnstcnt}
\newcommand{\newconstant}{%
\refstepcounter{cnstcnt}%
\ensuremath{c_{\thecnstcnt}}}
\newcommand{\oldconstant}[1]{\ensuremath{c_{\ref{#1}}}}

\tikzcdset{scale cd/.style={every label/.append style={scale=#1},
    cells={nodes={scale=#1}}}}

\begin{document}
\title[Plaquette Random-Cluster Model \& Potts Lattice Gauge Theory]{Topological Phases in the \\ Plaquette Random-Cluster Model \\ and Potts Lattice Gauge Theory}

\author{Paul Duncan}
\email{paul.duncan@mail.huji.ac.il}
\thanks{P.D.\ gratefully acknowledges the support of NSF-DMS \#1547357.}
\address{Einstein Institute of Mathematics, Hebrew University of Jerusalem, Jerusalem 91904, Israel}
\author{Benjamin Schweinhart}
\email{bschwei@gmu.edu}
\address{Department of Mathematical Sciences, George Mason University, Fairfax, VA 22030, USA}

\maketitle

\begin{abstract}
The $i$-dimensional plaquette random-cluster model on a finite cubical complex is the random complex of $i$-plaquettes with each configuration having probability proportional to
$$p^\text{\# of plaquettes}\paren{1-p}^\text{\# of complementary plaquettes}q^{\mathbf{b}_{i-1}}\,,$$
where $q\geq 1$ is a real parameter and $\mathbf{b}_{i-1}$ denotes the rank of the $(i-1)$-homology group with coefficients in a specified coefficient field. When $q$ is prime and the coefficient field is $\F_q$, this model is coupled with the $(i-1)$-dimensional $q$-state Potts lattice gauge theory. We prove that the probability that an $(i-1)$-cycle in $\Z^d$ is null-homologous in the plaquette random-cluster model equals the expectation of the corresponding generalized Wilson loop variable. This provides the first rigorous 
justification for a claim of Aizenman, Chayes, Chayes, Fr\"olich, and Russo that there is an exact relationship between Wilson loop variables and the event that a loop is bounded by a surface in an interacting system of plaquettes. 
We also prove that the $i$-dimensional plaquette random-cluster model on the $2i$-dimensional torus exhibits a sharp phase transition at the self-dual point $p_{\mathrm{sd}} \coloneqq \frac{\sqrt{q}}{1+\sqrt{q}}$ in the sense of homological percolation. This implies a qualitative change in the generalized Swendsen--Wang dynamics from local to non-local behavior.  
\end{abstract}

\section{Introduction}

The random-cluster model~\cite{fortuin1972random, grimmett2006random}, also known as the Fortuin--Kasteleyn model, is a random subgraph of a finite graph where the probability of a given configuration is proportional to
$$p^\text{\# of edges}\paren{1-p}^\text{\# of complementary edges}q^\text{\# of connected components}.$$
This model can then be extended to infinite graphs via limits on finite subgraphs when $q \geq 1.$ The random cluster model is an important tool for the study of the Potts model of interacting spins on the vertices of a graph. Together, they are the subject of an expansive literature in mathematical physics, probability, and statistical physics. The relationship between the models is exemplified by the connection--correlation theorem, which expresses the correlation between the spins of two vertices in the Potts model in terms of the probability that they are in the same connected component in the corresponding random cluster model.   

One-dimensional $q$-state Potts lattice gauge theory is a higher-dimensional analogue of the Potts model, where spins are assigned to edges of a graph rather than to vertices. It was as introduced in the physics literature~\cite{wegner1971duality, kogut1980z} as a relatively approachable analogue of Euclidean lattice gauge theory.  Lattice gauge theory is in turn a discretized models of Euclidean Yang--Mills theory~\cite{wilson1974confinement,  wilson2004origins, chatterjee2016yang}. The most important observables for these models are the Wilson loop variables, which are the sum of the spins around a loop of edges (or, equivalently, the evaluation of the co-chain on an $1$-cycle)\footnote{We view a Wilson loop variable as a sum rather than as a product to match the notation in algebraic topology. This point will be discussed more below}.

In their celebrated paper on independent plaquette percolation,  Aizenman, Chayes, Chayes, Fr\"olich, and Russo posited the existence of an exact relationship between Wilson loop variables and the event that a loop is bounded by a surface in an interacting system of plaquettes, or random $2$-complex~\cite{ACCFR83}, in some sense generalizing the connection-correlation theorem. Inspired by the random cluster model and its relationship with the Potts model, Maritan and Omero~\cite{maritan1982gauge} defined such a random $2$-complex associated to $1$-dimensional Potts lattice gauge theory. Their goals were to to make rigorous the relationship between lattice gauge theory and the behavior of random surfaces and  to interpret the $q\rightarrow 1$ limit of Potts lattice gauge theory in terms of surfaces of independent plaquettes (plaquette percolation). However, their model was defined vaguely in terms of closed surfaces of plaquettes rather than homology.  This imprecision led to inconsistencies that were found by were found by Aizenman and Fr\"olich ~\cite{aizenman1984topological}: the Betti numbers depend on the coefficients\footnote{Note that $\mathbf{b}_{0}$ counts the number of connected components regardless of the coefficient field.}. 

In essence, Maritan and Omero defined a random $2$-complex $P$  so\looseness=-1
\[\mathbb{P}\paren{P}\propto p^{\text{\# of plaquettes}}\paren{1-p}^{\text{\# of complementary plaquettes}}q^{\mathbf{b}_{1}\paren{P;\,\F_2}}\,,\]
where the first Betti number $\mathbf{b}_{1}\paren{P;\,\F_2}$ counts the number of ``independent loops'' in $P,$ and is related to the number of ``independent closed surfaces'' by the Euler--Poincar\'{e} formula. Upon close inspection, this model is not coupled with Potts lattice gauge theory when $q\neq 2.$ An earlier paper of Ginsparg, Goldschmidt, and Zuber~\cite{ginsparg1980large} performed a series expansion for Potts lattice gauge theory which suggests this construction,  and also an alternative model where only oriented surfaces are taken into account. This has the effect of replacing $\mathbf{b}_{1}\paren{P;\,\F_2}$ with $\mathbf{b}_{1}\paren{P;\,\Z},$ yielding a random $2$-complex that is not coupled with Potts lattice gauge theory for any choice of $q.$ Nor is Potts lattice gauge theory coupled with a different generalized random-cluster model due to Chayes and Chayes~\cite{chayes1984correct} where $\mathbf{b}_{1}\paren{P;\,\F_2}$ is replaced by the the number of strongly connected components of the set of plaquettes. We show here that the ``correct'' approach is to view the plaquette random-cluster of Hiraoka and Shirai~\cite{hiraoka2016tutte} as a family with three parameters, namely $p$, $q$, and the coefficient field $\F.$ Then, by setting $\F=\F_q$ and choosing $p$ appropriately, we can express Wilson loop variables in terms of the probability that the loop is bounded by a surface of plaquettes, where the coefficient field delimits the set of admissible surfaces. Moreover, fixing $\F_q$ and taking the $q\rightarrow 1$ limit allows a comparison of $q$-state Potts lattice gauge theory with plaquette percolation.

It is conjectured that the expectations of Wilson loop variables in Potts lattice gauge theory exhibit a sharp phase transition from a ``perimeter law'' to an ``area law,'' in which they decay exponentially in the perimeter and area of the loop respectively. Our results show that this is equivalent to a corresponding conjecture for the event that the loop is null-homologous in the plaquette random cluster model. While these conjectures remains out of reach, we are able to prove a qualitative analogue in terms of homological percolation on the torus. That is, we show that the $i$-dimensional plaquette random-cluster model on the $2i$-dimensional torus exhibits a sharp phase transition at the self-dual point $p_{\mathrm{sd}} \coloneqq \frac{\sqrt{q}}{1+\sqrt{q}}$ marked by the the emergence of giant cycles which are non-trivial homology classes in the ambient torus. When $q$ is an odd prime, these cycles are ruled by Polyakov loops that are constant in the corresponding Potts lattice gauge theory. The critical point is alternatively characterized by a qualitative change in the generalized Swendsen--Wang dynamics from local to non-local behavior.

\subsection{Definitions, Conjectures, and Results}\label{sec:DCR}

The generalized random cluster model we study here was introduced by  Hiraoka and Shirai~\cite{hiraoka2016tutte}. 
Their model, which we call the $i$-dimensional plaquette random-cluster model, weights the probability of an $i$-complex $P$ in terms of the  $(i-1)$-Betti number of the homology with coefficients in a field $\F,$ denoted  $\mathbf{b}_{i-1}\paren{P;\,\F}.$  We provide a definition of Betti numbers in Section~\ref{sec:hom} below.

\begin{Definition} Let $X$ be a finite $d$-dimensional cell complex, let $0<i<d,$ and fix a field $\F$ and parameters $p\in\brac{0,1}$ and $q\in\paren{0,\infty}.$ The \emph{$i$-dimensional plaquette random-cluster model} on $X$ is the random $i$-complex $P$ containing the full $\paren{i-1}$-skeleton of $X$ with the following distribution:
\begin{align*}
    \mu_{X}\paren{P} = \mu_{X,p,q,i,\mathbb{F}}\paren{P} \coloneqq \frac{1}{Z}p^{\abs{P}}\paren{1-p}^{\abs{X^i} - \abs{P}}q^{\mathbf{b}_{i-1}\paren{P;\,\mathbb{F}}}\,,
\end{align*}
where $Z=Z\paren{X,p,q,i,\F}$ is a normalizing constant, and $\abs{X^i}$ and $\abs{P}$ denote the number of $i$-cells of $X$ and $P,$ respectively.
\end{Definition}
As $\mathbf{b}_0$ is the number of connected components, this is indeed a generalization of the classical random-cluster model. In their paper introducing the plaquette random-cluster model, Hiraoka and Shirai's main focus is proving an expression for the partition function in terms of a generalized Tutte polynomial. They also show positive association and construct a coupling between the model and a generalization of the Potts model. However, they do not mention that this generalized Potts model had been previously defined as Potts lattice gauge theory~\cite{hiraoka2016tutte}.

Generalized $i$-dimensional Potts lattice gauge theory assigns spins in an abelian group $\mathcal{G}$ to the $i$-dimensional faces of an oriented cell complex $X.$ Denote by $C^i\paren{X;\,\mathcal{G}}$ the set of functions from the oriented $i$-cells of $X$ to  $\mathcal{G}$ so that cells with opposite orientations are mapped to inverse group elements. We follow the language of algebraic topology and call elements of $C^i\paren{X;\,\mathcal{G}}$ cochains. In other works on lattice gauge theory they are sometimes termed discrete differential forms, for reasons that we will describe in Section~\ref{sec:hom} below. 
\begin{Definition}
Let $\mathcal{G}$ be a finite abelian group and let $X$ be a finite cubical complex. For $f\in C^{i}\paren{X;\,\mathcal{G}}$ let
\begin{equation}
    \label{eq:hamiltonian_potts}
    H\paren{f}=-\sum_{\sigma}K\paren{\delta f\paren{\sigma},1}\,,
\end{equation}
where $\delta$ is the coboundary operator $\delta f\paren{\sigma}=f\paren{\partial \sigma}$ and $K$ is the Kronecker delta function. \emph{$(i-1)$-dimensional Potts lattice gauge theory} on $X$ with coefficients in $\mathcal{G}$ is the measure
\[\nu_{X,\beta,\mathcal{G},i-1,d}\paren{f}=\frac{1}{\mathcal{Z}} e^{-\beta H\paren{f}}\]
where $\beta$ is a parameter (called the coupling constant or inverse temperature) and $\mathcal{Z}=\mathcal{Z}\paren{X,\beta,\mathcal{G},i-1,d}$ is a normalizing constant.
\end{Definition}
In this paper, we focus exclusively on the cases where $\mathcal{G}=\F_q$ is the additive group of integers modulo a prime number $q$ (which can be identified with the multiplicative group of $q$-th complex roots of unity $\Z\paren{q}$). In keeping with conventions from algebraic topology, we use additive notation rather than the multiplicative notation of the definition of Potts lattice gauge theory. In particular, we will distinguish between the additive group of integers modulo $n$ $\Z_n$ and the multiplicative group of complex $n$-th roots of unity $\Z(n),$ though they are isomorphic. In the case where $\mathcal{G}=\F_q,$ we will denote the Potts lattice gauge theory measure by $\nu_{\beta,q,i-1,d},$ or simply $\nu,$ and call this measure $q$-state Potts lattice gauge theory. We describe why the assumption that $q$ is prime is necessary in Section~\ref{sec:prime}.  We will be most interested in the Potts lattice gauge theory on $\Z^d,$ and we will write $\nu^{\mathbf{f}}=\nu^{\mathbf{f}}_{\Z^d,\beta,\mathcal{G},i-1,d}$ to denote the limiting Potts lattice gauge theory with free boundary conditions constructed in Section~\ref{sec:potts_infinitevolume}.

Potts lattice gauge theory is motivated by its similarities with Euclidean lattice gauge theory.  Lattice gauge theory is in turn a  discretized model of Euclidean Yang--Mills theory~\cite{wilson1974confinement, wilson2004origins, chatterjee2016yang}.  Very briefly, Euclidean lattice gauge theory is defined so that $\F_q$ can be replaced with a complex matrix group $\mathcal{G}.$ In particular, $1$-dimensional lattice gauge theory on $\Z^4$ with $\mathcal{G}=U(1),$ $\mathcal{G}=SU(2),$ and $\mathcal{G}=SU(3)$, models the electromagnetic, weak nuclear, and strong nuclear forces, respectively. While Potts lattice gauge theory is not itself physical, it has been studied in the physics literature as it is more tractable and is thought to present some of the same behavior as more physically relevant cases~\cite{ginsparg1980large, kotecky1982first, maritan1982gauge, aizenman1984topological, laanait1989discontinuity}. The special cases of $2$ and $3$-state Potts lattice gauge theory coincide with $\Z(2)$ (Ising) and $\Z(3)$ ``clock'' lattice gauge theory respectively after an appropriate rescaling of the coupling constant $\beta$ (where $\Z(n)$ denotes the multiplicative group of $n$-th complex roots of unity). These models have also  been studied in the physical and the mathematical literatures as relatively approachable examples of lattice gauge theory~\cite{altes1978duality, yoneya1978z, frohlich1979confinement, mack1979comparison, marra1979statistical,  kogut1980z,  wipf2007generalized, wegner2017duality,  chatterjee2020wilson}.

The most important observables in lattice gauge theory are arguably Wilson loop variables, which we define for $q$-state Potts lattice gauge theory. These arise from evaluating a Potts state $f$ on a ``loop'' $\gamma$ made of $(i-1)$-cells. In topological language $\gamma \in C_{i-1}\paren{X;\;\mathcal{G}}$ is called an $(i-1)$-cycle.

\begin{Definition}
Let $f$ be an $(i-1)$-cocycle. The \emph{generalized Wilson loop variable} associated to $\gamma$ is the random variable $W_{\gamma}:C^{i-1}\paren{X;\,\F_q}\rightarrow \C$ given by 
\[W_{\gamma}\paren{f}=\paren{f\paren{\gamma}}^{\mathbb{C}}\,,\]
where the $\C$ superscript denotes that we are viewing the variable as a complex number. That is, if $g\in \F_q,$ $g^\C$ is the corresponding $q$-th root of unity in $\C.$ 
\end{Definition}
An important conjecture for the analogous quantities in Euclidean lattice gauge theory --- called the Wilson area law --- is that if $\gamma=\partial\rho$ is a $1$-boundary and $\rho$ is the minimal bounding chain then the expectation of $W_{\partial\rho}\paren{\omega}$ should decay as $e^{-c\abs{\rho}}$ in some cases. For $\mathcal{G}=SU(2)$ or $SU(3),$ this conjecture is thought to be related to the phenomenon of quark confinement: that charged particles for the weak or strong nuclear forces are not seen in isolation, unlike for for the electromagnetic force~\cite{wilson1974confinement,chatterjee2021probabilistic} where the corresponding $U\paren{1}$ lattice gauge theory exhibits both ``area law'' and ``perimeter law'' phases~\cite{guth1980existence,frohlich1982massless}.

One-dimensional $q$-state Potts lattice gauge theory on $\Z^4$ is thought to undergo a sharp phase transition as the coupling constant $\beta$ increases, from a ``perimeter law''  to an``area law'' regime. We state a more general form of this ``sharpness'' conjecture here.
  \begin{Conjecture}
  \label{conj:sharpness}
Fix integers $q \geq 2$ and $0\leq i<d.$ Then there are constants $0<\beta_c\paren{q}<\infty$, and $0<\newconstant\label{const:conj1}(\beta,q),\newconstant\label{const:conj2}(\beta,q)<\infty$ so that, for hyperrectangular $(i-1)$-boundaries $\gamma$ in $\Z^d,$
\begin{align*}
    -\frac{\log\paren{\mathbb{E}_{\nu^\mathbf{f}}(W_\gamma)}}{\mathrm{Area}(\gamma) } \rightarrow &   \oldconstant{const:conj1}(\beta,q) \qquad && \beta < \beta_{c}(q)\\
   -\frac{\log\paren{\mathbb{E}_{\nu^\mathbf{f}}(W_\gamma)}}{ \mathrm{Per}(\gamma)} \rightarrow &   \oldconstant{const:conj2}(\beta,q) \qquad && \beta > \beta_{c}(q)
      \,,
\end{align*}
as all dimensions of $\gamma$ are taken to $\infty.$ In addition, if $d=2i,$ then $\beta_c\paren{q}=\beta_{\mathrm{sd}}(q)=\log\paren{1+\sqrt{q}}.$
\end{Conjecture}
  
Recall that $W_\gamma$ is defined as a complex number, so the expectation is complex as well. Here, the \emph{perimeter} of an $(i-1)$-boundary $\gamma$ is the number of plaquettes in its support and its \emph{area} is the number of plaquettes supported in the minimal bounding chain. Observe that when $i=1,$ $\gamma$ consists of two vertices $\set{v,w},$ its perimeter is $2$, and its area is the distance between $v$ and $w.$ As such, the $i=1$ case of the conjecture is the the sharpness of the phase transition for the Potts model, as proven by Duminil-Copin, Raoufi, and Tassion~\cite{duminil2019sharp}.

This conjecture is not known for any specific value of $q$ when $i>1$, but Laanait, Messager, Ruiz showed that such a transition occurs when $q$ is sufficiently large~\cite{laanait1989discontinuity}. In addition, classical series expansions have been used to prove the existence of perimeter law and area law regimes for sufficiently extreme values of $\beta.$ Specifically, Osterwalder and Seiler employed a ``high--temperature expansion'' to show existence of an area law regime for any lattice gauge theory~\cite{osterwalder1978gauge, seiler1982gauge}, and the argument should work for Potts lattice gauge theory as well. On the other hand, a perimeter law regime is demonstrated by a ``low--temperature expansion,'' a technique that has recently been used by Chatterjee~\cite{chatterjee2020wilson}, Cao~\cite{cao2020wilson}, and Forsström, Lenells, and Viklund~\cite{forsstrom2020wilson}  to prove precise estimates for the asymptotics of Wilson loop variables for lattice gauge theory when $\beta$ is large. We use the coupling of the plaquette random cluster model and Potts lattice gauge theory to give an alternate proof of the existence of perimeter law and area law regimes for $i$-dimensional $q$-state Potts lattice gauge theory in $\Z^d,$ for prime values of $q.$ Towards that end, we prove an exact relationship between Wilson loop variables and a certain topological event.

An $i$-chain $\gamma$ is $\emph{null-homologous}$ in a cubical complex $P$ if $\gamma=\partial \rho$ for some $(i+1)$-chain $\rho\in C_{i+1}\paren{P;\,\mathcal{G}}$ (in other words, $\gamma=0$ the $i$-dimensional homology group).  Roughly speaking, when $i=1$ this means that $\gamma$ is ``bounded by a surface of plaquettes,'' with the surfaces under consideration being dependent on the coefficient group $\mathcal{G}.$ For $\mathcal{G}=\Z,$ $\gamma$ is null-homologous if and only if it is bounded by an orientable surface of plaquettes; when $\mathcal{G}=\F_2$ any surface of plaquettes will suffice. We denote the event that $\gamma$ is null-homologous by $V_\gamma,$ with the choice of coefficients being understood in context. We have the following result.

  \begin{Theorem}
\label{thm:comparison}
Suppose $q \geq 2$ is a prime integer, let $0<i<d-1,$ and let $\gamma$ be an $(i-1)$-cycle in $\Z^d$.  Then
\[\mathbb{E}_{\nu^{\mathbf{f}}}\paren{W_{\gamma}}=\mu_{\Z^d}\paren{V_{\gamma}}\,,\]
where $\nu^{\mathbf{f}}=\nu_{\Z^d,\beta,q,i-1,d}^{\mathbf{f}}$ is Potts lattice gauge theory and $\mu_{\Z^d}=\mu_{\Z^d,1-e^{-\beta},q,i}$ is the corresponding random cluster model.
\end{Theorem}

While we state this theorem for $\Z^d,$ the same proof works for any finite cubical complex. A precise definition of the infinite volume $i$-dimensional plaquette random cluster model $\mu_{\Z^d}=\mu_{\Z^d,p,q,i}$ is given in Section~\ref{sec:RCM_infinitevolume}. As a corollary, we have that Conjecture~\ref{conj:sharpness} is equivalent to the special case of the following conjecture for the plaquette random cluster model when $q$ is a prime integer and $\F=\Z_q.$

  \begin{Conjecture}
  \label{conj:sharpnessRCM}
 There exist constants $p_c\paren{q}\in \paren{0,1}$ and  $0<\newconstant\label{const:conj3}\paren{p,q},\newconstant\label{const:conj4}\paren{p,q}<\infty$ so that, for hyperrectangular $(i-1)$-boundaries $\gamma$ in $\Z^d,$
 \begin{align*}
    -\frac{\log\paren{\mu_{\Z^d}(V_\gamma)}}{\mathrm{Area}(\gamma) } \rightarrow &   \oldconstant{const:conj3}(p,q) \qquad && p < p_c(q)\\
   -\frac{\log\paren{\mu_{\Z^d}(V_\gamma)}}{ \mathrm{Per}(\gamma)} \rightarrow &   \oldconstant{const:conj4}(p,q) \qquad && p > p_c(q)
      \,,
\end{align*}
as all dimensions of $\gamma$ are taken to $\infty.$  Also, when $d=2i,$ $p_c\paren{q}=\frac{\sqrt{q}}{1+\sqrt{q}}$.
\end{Conjecture}

When $i=1,$ this conjecture is the sharpness of the  phase transition of the classical random cluster model, as proven in~\cite{duminil2019sharp}. In addition, the special case of the conjecture for $2$-dimensional independent plaquette percolation in $\Z^3$ ($i=2,d=3,q=1$) is a theorem of Aizenman, Chayes, Chayes, Fr\"olich, and Russo, who demonstrated that this phase transition is dual to the bond percolation transition\footnote{Their paper phrases this result in terms of possibly distinct critical probabilities, but these have since been shown to coincide~\cite{menshikov1986coincidence,grimmett1990supercritical}.}. Their proof works for general coefficient fields.

Another consequence of Theorem~\ref{thm:comparison}, together with a comparison with plaquette percolation, is a new proof of the existence of area law and perimeter law regimes for Potts lattice gauge theory. We state a more general result for the plaquette random cluster model.

\begin{Theorem}
\label{thm:areaperimeter}
Let $\F$ be a field, $q\in \left[1,\infty\right)$ and let $0<i<d.$ There exist positive, finite constants $\newconstant\label{const:3}=\oldconstant{const:3}(p,q,i,d,\mathbb{F}),\newconstant\label{const:4}=\oldconstant{const:4}(p,q,i,d,\mathbb{F})$ and $0<p_1\leq p_2<1$ so that, for hyperrectangular $(i-1)$-boundaries $\gamma$ in $\Z^d,$
\begin{equation}
    \label{eq:wilsoninequalities1}
 \exp(-\oldconstant{const:3} \mathrm{Area}(\gamma)) \leq \mu_{\Z^d}(V_\gamma) \leq  \exp(-\oldconstant{const:4} \mathrm{Per}(\gamma))
\end{equation}
for all $p \in \paren{0,1}$ and so that
\begin{align*}
&-\frac{\log\paren{ \mu_{\Z^d}(V_\gamma)}}{\mathrm{Per}(\gamma))}\;\;=\;\; \Theta\paren{1} &&\quad \text{if } p < p_1\\
&-\frac{\log\paren{ \mu_{\Z^d}(V_\gamma)}}{\mathrm{Area}(\gamma)}\;\; \rightarrow \;\; \oldconstant{const:3} && \quad \text{if } p>p_2\,.
\end{align*}
\end{Theorem}
Note that the constants could depend on the boundary conditions used to construct the infinite volume plaquette random cluster model.

  The perimeter law phase can be interpreted as a ``hypersurface--dominated regime'' in a quantitative sense, where there are many large hypersurfaces of $i$-plaquettes in $\Z^d$ \cite{ACCFR83}. For plaquette percolation on the torus, Duncan, Kahle, and Schweinhart~\cite{duncan2020homological} showed that this phase coincides with a ``hypersurface--dominated regime'' in the qualitative sense of homological percolation, that is, one marked by the appearance of $i$-dimensional cycles in the random cubical complex that are representatives of nonzero homology classes in the ambient torus (see Section~\ref{sec:hom} for more detail).  Unlike the theorem of~\cite{ACCFR83}, this homological percolation result generalizes to to $i$-dimensional plaquette percolation in the $2i$ dimensional torus. Here, we further generalize it to the $i$-dimensional random-cluster model in the $2i$-dimensional torus.
  
\begin{Theorem}
\label{thm:half}
Suppose $d=2i,$ and let $\phi_{*} : H_i\paren{P;\,\F} \to H_i\paren{\T^d_N;\,\F}$ be the induced map on homology with coefficients in a field $\F$ with $\mathrm{char}\paren{F} \neq 2.$ Denote by $A$ and $S$ the events that $\phi_*$ is non-trivial and surjective, respectively, and set $p_{\mathrm{sd}} = \frac{\sqrt{q}}{1+\sqrt{q}}.$ Then
\[\begin{cases}
\mu_{\T^d_N}\paren{A}\rightarrow 0 & p<p_{\mathrm{sd}}(q) \\
\mu_{\T^d_N}\paren{S}\rightarrow 1 & p>p_{\mathrm{sd}}(q) \\
\end{cases}\]
as $N\rightarrow \infty.$
\end{Theorem}

When $q$ is a prime integer and $i=2$, the phase transition is marked by the emergence of ``giant sheets'' on which the Wilson loop variables are constant within homology classes. The critical point is alternatively characterized by a qualitative change in the generalized Swendsen--Wang dynamics from local to non-local behavior, as defined in Section~\ref{sec:sw} below.

We also generalize the other two main theorems of~\cite{duncan2020homological} from plaquette percolation to the plaquette random-cluster model. The proofs of these results and the previous one rely heavily on the duality properties of the plaquette random-cluster model; in Theorem~\ref{theorem:duality} we show that the dual of $\mu_{\T^d_N,p,q,i}$ is distributed approximately as $\mu_{\T^d_N,p^*,q,d-i}$ in a precise sense, where 
\[p^*=p^*(p,q)= \frac{\paren{1-p}q}{\paren{1-p}q + p}\,.\]

Next, we prove that there are dual sharp phase transitions for $i=1$ and $i=d-1$ that are consistent with the critical probability for the random-cluster model in $\Z^d,$ assuming a conjecture about the continuity of the critical probability in slabs. Let 
\[S_k \coloneqq \Z^2 \times \set{-k,-k+1,\ldots,k}^{d-2} \subset \Z^d\,.\]
Fix $q \geq 1$ and let $p_c\paren{S_k}$ be the critical probability for the $1$-dimensional random-cluster model with parameter $q$ on $S_k$ with free boundary conditions. This can be constructed by a limit of free random-cluster models on 
\[S_{k,l} \coloneqq \set{-l,-l+1,\ldots,l}^2 \times \set{-k,-k+1,\ldots,k}^{d-2}\,.\]
Since $S_{k,l} \subset S_{k+1,l},$ it follows that $p_c\paren{S_k}$ is decreasing in $k.$ Then let \[p_c^{\mathrm{slab}} = \lim_{k \to \infty} p_c\paren{S_k}\,.\]

Let $\hat{p}_c=\hat{p}_c(q,d)$ be the critical threshold for the random-cluster model with parameter $q$ on $\Z^d.$ These two critical values are conjectured to coincide.

\begin{Conjecture}[\cite{pisztora1996surface}]\label{con:slab}
For all $q \geq 1,$
\[p_c^{\mathrm{slab}} = \hat{p}_c\,.\]
\end{Conjecture}

\begin{Theorem}
\label{thm:one}
Let $q \geq 1.$ Then the following statements hold:

If $i=1$ then
\[\begin{cases}
\mu_{\T^d_N}\paren{A}\rightarrow 0 & p<\hat{p}_c \\
\mu_{\T^d_N}\paren{S}\rightarrow 1 & p>p_c^{\mathrm{slab}} \\
\end{cases}\]
as $N\rightarrow \infty.$

Furthermore, if $i=d-1$ then 
\[\begin{cases}
\mu_{\T^d_N}\paren{A}\rightarrow 0 & p<(\hat{p}_c)^* \\
\mu_{\T^d_N}\paren{S}\rightarrow 1 & p>(p_c^{\mathrm{slab}})^* \\
\end{cases}\]
as $N\rightarrow \infty.$

In particular, if Conjecture~\ref{con:slab} is true, then there are sharp thresholds at $\hat{p}_c = p_c^{\mathrm{slab}}$ and $\paren{\hat{p}_c}^* = \paren{p_c^{\mathrm{slab}}}^*$ respectively.
\end{Theorem}

In general dimensions we demonstrate the existence of a sharp threshold function defined as follows: Let $\lambda = \lambda\paren{N,q,i,d}$ satisfy
\begin{align*}
    \mu_{\T^d_N,\lambda,q,i}\paren{A} = 1/2
\end{align*}
and let $p_l = p_l\paren{q,i,d} \coloneqq \liminf_{N \to \infty} \lambda\paren{N,q,i,d}$ and $p_u = p_u\paren{q,i,d} \coloneqq \limsup_{N \to \infty} \lambda\paren{N,q,i,d}.$

\begin{Theorem}
\label{thm:weak}

Let $q \geq 1$ be a prime integer and suppose that $\mathrm{char}\paren{F} \neq 2.$ For every $d \ge 2,$ $1 \le i \le d-1,$ and $\epsilon > 0$
\[\begin{cases}
\mu_{\T^d_N,\lambda-\epsilon,q,i}\paren{A}\rightarrow 0 \\
\mu_{\T^d_N,\lambda+\epsilon,q,i}\paren{S}\rightarrow 1 \\
\end{cases}\]
as $N\rightarrow \infty.$

Moreover, for every $d \ge 2$ and $1 \le i \le d-1$ we have 
\[0<p_l \leq p_u<1\,,\]
and $p_l$ and $p_u$ have the following properties.
\begin{enumerate}[label=(\alph*)]
    \item (Duality) $p_u\paren{q,i,d} = \paren{p_l\paren{q,d-i,d}}^*.$
    \item (Monotonicity in $i$ and $d$) $p_u\paren{q,i,d} < p_u\paren{q,i,d-1}< p_u\paren{q,i+1,d}$ 
    for $0<i<d-1.$ 
\end{enumerate}
\end{Theorem}

\subsection{Why does $q$ need to be prime?}
\label{sec:prime}

Our results for $q$-state Potts lattice gauge theory are for cases where the cohomology coefficients are a finite field $\F_q$ where $q$ is a prime integer, instead of the more general additive group $\Z_n.$ This is not just for convenience, though it is more convenient to work with homology and cohomology with field coefficients since the associated groups are then vector spaces. There is one key place where this assumption is necessary, as noticed by Aizenman and Fr\"olich~\cite{aizenman1984topological}. It arises trying to couple the plaquette random-cluster model with Potts lattice gauge theory. Recall that the definition of the random-cluster model involves the term $q^{\mathbf{b}_{i-1}\paren{P}},$ where $\mathbf{b}_{i-1}\paren{P}$ is dimension of the $\F$-vector space $H_{i-1}\paren{P;\,\F}$. If we were to use $\Z_n$ coefficients, we could try to replace $\mathbf{b}_{i-1}\paren{P}$ with the rank of the $\Z_n$-module $H_i\paren{P;\,\Z_n}$ (that is, the size of the minimal generating set). However, this does not lead to the same model as the marginal of the generalized Edwards--Sokal coupling defined in Section~\ref{sec:latticegauge}. The argument in Proposition~\ref{prop:coupling} breaks down because, in general, 
\[\abs{H_i\paren{P;\,\Z_n}} \neq n^{\mathrm{rank}\paren{H_i\paren{P;\,\Z_n}}}\,.\]
This could be resolved by defining a dependent plaquette percolation model so that 
\[\mathbb{P}\paren{P}\propto p^{\abs{P}}\paren{1-p}^{\abs{X^i} - \abs{P}}\abs{H_i\paren{P;\,\Z_n}}\,.\]
We expect some of our results to carry over to this model, but we defer them to a later paper.

 Aizenman and Fr\"olich identified a second issue, where they constructed a plaquette system in which there exist $(i-1)$-cycles $\gamma$ and $\gamma'$ so that $\brac{\gamma}$ is homologous to $q\brac{\gamma'}.$ Then an analogue of Theorem~\ref{thm:comparison} fails when $V_{\gamma}$ is defined to be the event that $\gamma$ is bounded by a surface of plaquettes. However, we resolve this by redefining $V_{\gamma}$ to depend on the choice of $q$; clearly $\brac{\gamma}=0$ in $H_{i-1}\paren{P;\;\Z_q}$ in this example.
 

Finally, in the proof of the homological percolation phase transition, Lemma~\ref{lemma:spinning} requires that the homology groups are vector spaces.

\subsection{Outline} We outline the remainder of the paper. First, in Section~\ref{sec:hom} reviews the definitions of homology and cohomology for cubical complexes. These are used in Section~\ref{sec:topdual} to show useful topological duality results. Next, we define the plaquette random-cluster model in Section~\ref{sec:RCM} and prove several of its basic properties. These include positive association (which was previously shown by Hiraoka and Shirai~\cite{hiraoka2016tutte}), the existence of the infinite volume limit, and duality. Section~\ref{sec:latticegauge} covers the relationship between the plaquette random-cluster model and Potts lattice gauge theory: the coupling between them, the definition of the plaquette Swendsen--Wang algorithm, and proofs of Theorems~\ref{thm:comparison} and~\ref{thm:areaperimeter}. Finally, we prove our results for homological percolation --- Theorems~\ref{thm:half}, \ref{thm:one}, and \ref{thm:weak} --- in Section~\ref{sec:homperc}.

\section{Homology and Cohomology}\label{sec:hom}

We give a brief review of homology and cohomology, two fundamental invariants studied in algebraic topology. These theories exhibit a number of dualities which are closely related to classical dualities between lattice spin models and --- as we will demonstrate below --- allow them to be generalized to higher dimensions. Though this paper is not self-contained with respect to topological background, we hope to provide enough context for a reader with minimal previous knowledge of the area to make sense of our results. More detail can be found in~\cite{hatcher2002algebraic} (the standard mathematical reference) or~\cite{druhl1982algebraic} which provides an exposition specific to the context of lattice spin models but does not cover all the material we need here. Those familiar with differential forms but not algebraic topology may find some of these concepts familiar, as they are important in a continuous version of the discrete theories developed here.

In this paper we will deal exclusively with spaces called cubical complexes, which are composed of $i$-dimensional plaquettes (unit $i$-dimensional cubes) for various $i.$ Two references with details specific to this setting are \cite{kaczynski2004computational,saveliev2016topology}. However, many of the definitions and results in this paper can easily be adapted to more general cell complexes. A first example of a cubical complex is the integer lattice cubical complex $\Z^d,$ which is a union of all $i$-plaquettes for $0 \leq i \leq d$ which have corners in the integer lattice. The $d$-dimensional torus also has a natural cubical complex structure obtained by dividing the cube $\brac{0,N}^d$ into unit cubes and identifying opposite faces of the original cube. We call this cubical complex $\T_N^d.$ All examples we will consider in this paper are subcomplexes of either $\Z^d$ or $\T_N^d.$

\subsection{Homology}

In order to convert vague geometric questions about ``the number of holes'' of a topological space into concrete algebraic quantities, homology theory defines a space  $C_i\paren{X;\,\mathcal{G}}$ of linear combinations with coefficients in an abelian group $\mathcal{G}$ (for a concrete example, take $\mathcal{G} = \Z$) of $i$-plaquettes called \emph{chains}, and \emph{boundary operators}, which map an $i$-plaquette to a sum of its $(i-1)$-faces. The continuous analogues of these concepts are, roughly speaking ``$i$-dimensional spaces on which one can integrate an $i$-form'' (called an $i$-current) and the geometric boundary of a space.

For reasons that become apparent below, it is important to give each term in the boundary of a plaquette a sign corresponding to the orientation of a plaquette. The formula is relatively simple in low dimensions. A zero-plaquette is a vertex and its boundary is zero. A $1$-plaquette is an edge $\paren{v_1,v_2}$ and its boundary is the difference $v_2-v_1.$ A two-plaquette is an oriented unit square with vertices $\paren{v_1,v_2,v_3,v_4}$ and its boundary is $\paren{v_1,v_2}+\paren{v_2,v_3}+\paren{v_3,v_4}-\paren{v_1,v_4}.$ More generally, let $1 \leq k_1<k_2<\ldots<k_i \leq d$ and let $I_j = \brac{0,1}$ for $j \in \set{k_1,k_2,\ldots,k_i}$ and $I_j = \set{0}$ for $j \in \brac{d} \setminus {k_1,k_2,\ldots,k_i}.$ Then $\sigma = \prod_{1 \leq j \leq d} I_j$ is an $i$-plaquette in $\R^d,$ and its boundary is given by
\[\partial \sigma = \sum_{l=0}^i \paren{-1}^{l-1} \paren{\prod_{1 \leq j < k_l} I_j \times \set{1} \times \prod_{k_l < m \leq d} I_m - \prod_{1 \leq j < k_l} I_j \times \set{0} \times \prod_{k_l < m \leq d} I_m}\,.\]
The reason that the sum is alternating in sign is so that the boundary operator satisfies the equation
\begin{equation}\label{eq:boundaryofboundary}
    \partial \circ \partial = 0\,.
\end{equation}

\begin{figure}[t]
    \centering
    \includegraphics[width=.5\textwidth]{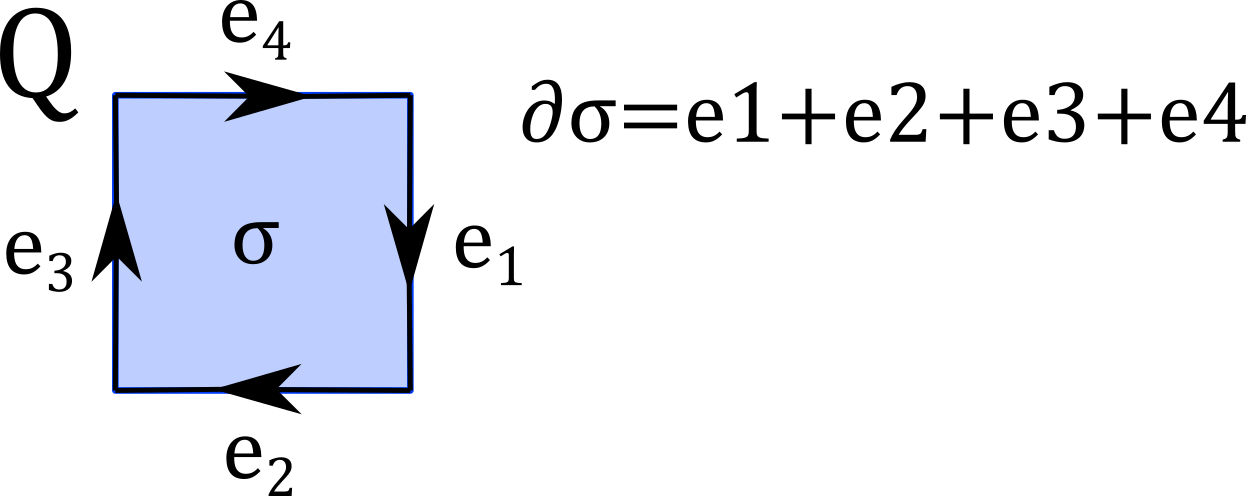}
    \caption{The boundary map for a two dimensional plaquette.}
    \label{fig:oneplaquette}
\end{figure}

Of particular interest are the chains $\alpha \in C_i\paren{X;\,\mathcal{G}}$ satisfying $\partial \alpha = 0.$ Such chains are called \emph{cycles}, the space of which is denoted $Z_i\paren{X;\,\mathcal{G}}.$ Equation~\ref{eq:boundaryofboundary} provides one source of $i$-cycles, namely the boundaries of $(i+1)$-plaquettes. We denote this space $B_{i}\paren{X;\,\mathcal{G}}.$ It turns out that the most interesting cycles are the ones that are not boundaries. 

To illustrate why this is the case, consider the unit square $Q = \brac{0,1}^2 \subset \R^2,$ seen in Figure~\ref{fig:oneplaquette}. In the cubical complex structure we defined earlier, we have that $Z_1\paren{Q;\,\mathcal{G}} \simeq \mathcal{G}$ consists of multiples of a single $1$-cycle, namely the boundary of $Q.$ But now imagine that we make our lattice spacing $1/2$ instead of $1.$ Now $Z_1\paren{Q;\,\mathcal{G}}$ contains linear combinations of the boundaries of the four $2$-plaquettes contained in $Q.$ One such combination can be seen in Figure~\ref{fig:fourplaquettes}. Though $Q$ has not changed, our choice of lattice has made the cycle space significantly more complicated.

\begin{figure}[t]
    \centering
    \includegraphics[width=.5\textwidth]{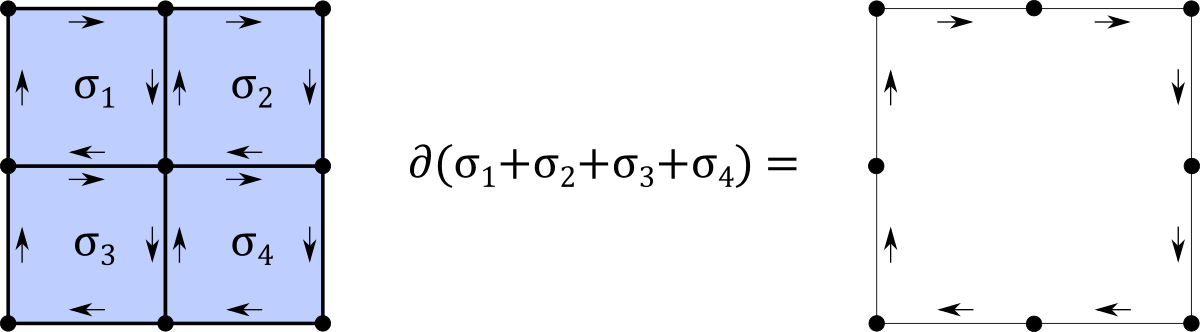}
    \caption{The boundary of a union of 4 plaquettes computed with the linearity of the boundary operator.}
    \label{fig:fourplaquettes}
\end{figure}

Therefore, in order to measure the shape of spaces in a way that does not depend on our choice of lattice, we define the \emph{$i$-dimensional homology group} as the quotient 
\[H_i\paren{X;\,\mathcal{G}} = Z_i\paren{X;\,\mathcal{G}}/B_i\paren{X;\,\mathcal{G}}\,.\] 
Returning to the unit square $Q,$ every $1$-cycle is a boundary regardless of the cubical complex structure on $Q$, so $H_1\paren{Q;\,\mathcal{G}} = 0.$ In general, it turns out that homology groups do not depend on the cubical complex structure of a space, and are invariant to continuous deformations of the space. In particular,
one can show that $H_1\paren{Q;\,\mathcal{G}} = 0$ by continuously deforming $Q$ to a point, which has no nonzero $1$-chains let alone $1$-cycles.

For an example with nontrivial first homology, consider $H_1\paren{\partial Q;\,\mathcal{G}},$ where $\partial Q$ is the topological boundary of $Q,$ i.e. the empty square formed by the union of the four sides of $Q.$ There is still the $1$-cycle from before but there are no $2$-plaquettes, so $H_1\paren{\partial Q;\,\mathcal{G}} \neq 0.$ This example illustrates the often repeated informal description of $H_i\paren{X}$ as a measurement of the number of $i$-dimensional holes of $X.$ In the case of $H_1,$ the cycles that remain are the loops that cannot be filled in by $2$-plaquettes.

We can now define giant cycles in a subcomplex of the torus. If $X \subset \T^d_N$ consists of a subset of the cubical cells of $\T^d_N,$ then the chain groups of $X$ are subgroups of the chain groups of $\T^d_N.$ The \emph{giant cycles} of $X$ are the cycles that are not boundaries when considered as chains in $\T^d_N.$ In the case $i=1,$ a giant cycle is a loop that spans the torus $\T^d_N$ in some sense, and must therefore have length at least $N.$ We interpret this as a finite volume analogue of an infinite path in the lattice. See Figure~\ref{fig:giantcycles} for illustrations of two giant cycles.

\begin{figure}
	\includegraphics[height=0.4\textwidth]{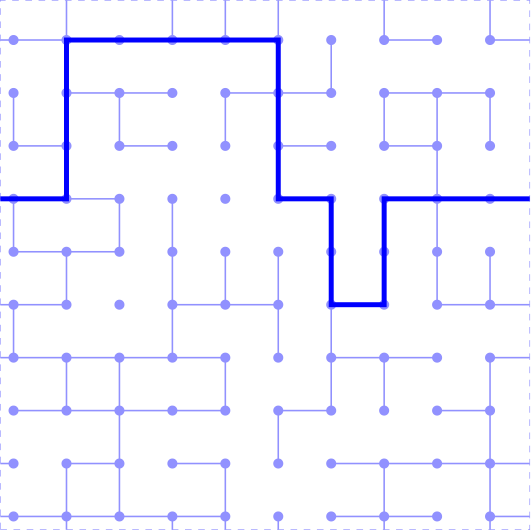}\qquad
	\includegraphics[height=0.45\textwidth]{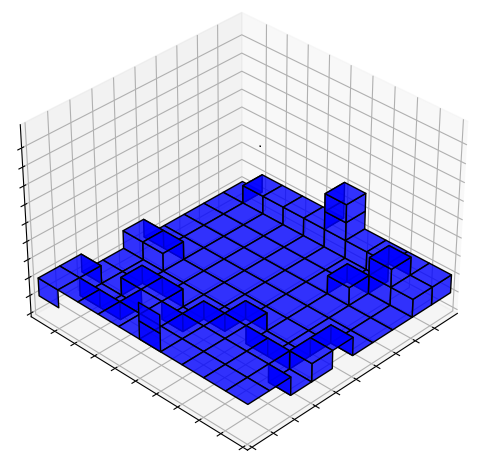}

    \caption{Examples of giant cycles: (left) a $1$-dimensional giant cycle in $\T^2$ and (right) a $2$-dimensional giant cycle in $\T^3,$ shown in cubes with periodic boundary conditions. Reproduced with permission from~\cite{duncan2020homological}.}
\label{fig:giantcycles}
\end{figure}

\subsection{Coefficients and Torsion}
So far, the choice of coefficient group $\mathcal{G}$ has not been important to our computations in this section, though we have mentioned their importance to the random-cluster model in the introduction. In this paper, we will almost exclusively work with homology over the finite field of integers mod $q$ for prime $q,$ denoted $\F_q.$ We will see later that this is the correct choice to study Potts lattice gauge theory, and it has the advantage of simplifying the algebra required since a homology group with field coefficients is a vector space.

For an example where different coefficients produce different results, consider the identifications of the sides of the $2$-plaquette in Figure~\ref{fig:torusklein}. In the first case we roll up the plaquette into a cylinder and then put the opposite ends of the cylinder together to form a torus $\T.$ In the second, we also form a cylinder but put the opposite ends of the cylinder together in the reverse orientation to form a Klein bottle $\mathbb{K}.$ Consider the second homology of both spaces. In both cases there are no $3$-plaquettes and only one $2$-plaquette $\sigma,$ so it suffices to check if the $2$-chains generated by $\sigma$ contain nontrivial cycles. In the torus, 
\[\partial \sigma = e_1 + e_2 - e_1 - e_2 = 0\,,\]
so $H_2\paren{\T;\,\F} \simeq \F$ for any field $\F.$ In the Klein bottle, 
\[\partial \sigma = e_1 + e_2 + e_1 - e_2 = 2e_1\,,\]
which is nonzero over most fields, with a number of exceptions including $\F_2.$ Thus, $H_2\paren{\mathbb{K};\,\F_2} \simeq \F_2,$ and $H_2\paren{\mathbb{K};\,\F_q} = 0$ for any prime $q \neq 2.$ There is also a difference in first homology; one can compute $H_1\paren{\mathbb{K};\,\F_2} \simeq \F_2 \oplus \F_2$ and $H_1\paren{\mathbb{K};\,\F_q} \simeq \F_q$ for $q \neq 2.$ Of course, there are no single plaquettes with sides identified this way in the integer lattice, but non-orientable surfaces such as the Klein bottle appear as unions of plaquettes in dimensions $4$ and higher. Homology that appears over some finite fields but not others is sometimes called \emph{torsion}.

\begin{figure}[ht]
    \centering
    \includegraphics[width=.5\textwidth]{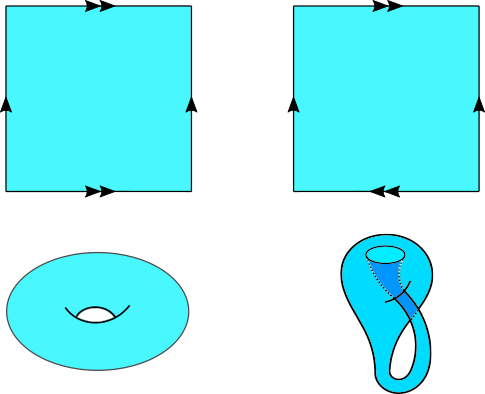}
    \caption{A torus $\T$ and a Klein bottle $\mathbb{K},$ the resulting shapes of two possible identifications of the opposite sides of a plaquette. The left plaquette is a cycle with any field coefficients and the right is only a cycle with $\F_2$ coefficients.}
    \label{fig:torusklein}
\end{figure}

It is worth mentioning a special situation that arises for the $0$-th homology group. It is not hard to check that $H_0\paren{X;\,\F} = \F^{{c_X}},$ where $c_X$ is the number of connected components of $X.$ Some standard results in algebraic topology have simpler statements when  all of the homology groups of a space consisting of a single point are zero. This is not true in the setup that we have described so far, since there is one connected component. To fix this, define the \emph{reduced homology groups} 
\[\Tilde{H}_k\paren{X;\,\F} = \begin{cases}
H_k\paren{X;\,\F} & k \geq 1\\
\F^{c_X-1}  & k = 0
\end{cases}\,.\]

\subsection{The  Euler--Poincar\'{e} formula} The  Euler--Poincar\'{e} formula is an important result that relates the Euler characteristic of a cubical complex to the ranks of its homology groups with coefficients in any field. Define the $i$-th Betti number of a cubical complex $X$ with coefficients in $\F$ by 
\[\mathbf{b}_i\paren{X} = \mathbf{b}_i\paren{X;\,\F} \coloneqq \dim H_i\paren{X;\,\F}.\]
Also, let $X^i$ be the set of $i$-plaquettes of $X$ and let $\abs{X^i}$ be their number. The Euler characteristic of $X$ is then
\[\chi\paren{X}=\sum_{i=0}^{\mathrm{dim}\paren{X}}\paren{-1}^{i}\abs{X^i}\,.\]
The following theorem is a straightforward consequence of the rank--nullity theorem.
\begin{Theorem}[Euler--Poincar\'{e} formula]
For any cubical complex $X$ and any field $\F$
\[\chi\paren{X}=\sum_{i=0}^{\mathrm{dim}\paren{X}}\paren{-1}^{i}\mathbf{b}_i\paren{X;\,\F} \,.\]
\end{Theorem}

\subsection{Cohomology}

A small algebraic change in the definitions that give the homology groups leads to cohomology. The $i$-th \emph{cochain group} $C^i\paren{X;\,\F}$ is defined as the group of $\F$-linear functions from $C_i\paren{X;\,\F}$ to $\F.$ Since we are using field coefficients, this is the dual vector space to $C_i\paren{X;\,\F}.$ Then the \emph{coboundary operator} $\delta : C^i\paren{X;\,\F} \to C^{i+1}\paren{X;\,\F}$ is defined by 
\[\delta f\paren{\alpha} = f\paren{\partial \alpha}\]
for $f\in C^i\paren{X;\,\F}, \alpha \in C_{i+1}\paren{X;\,\F}.$ The continuous analogues of an $i$-cochain and and its coboundary are a differential form and its exterior derivative. Note that if $f$ is an $i$-cochain and $\alpha$ is an $i$-chain, the evaluation  $f\paren{\alpha}$ corresponds to integration and the definition of the coboundary operator corresponds to Stokes' Theorem.

Analogously to homology groups, we define the $i$-th cocycle group $Z^i\paren{X;\,\F}$ to be the kernel of $\delta$ and the $i$-th coboundary group $B^i\paren{X;\,\F}$ to be the image of $\delta$ applied to $C^{i+1}\paren{X;\,F}.$ Then $i$-th cohomology group is the quotient \[H^i\paren{X;\,\F} = Z^i\paren{X;\,\F} / B^i\paren{X;\,\F}\,.\]
In the continuous setting, the analogue of the $i$-dimensional cohomology group is the $i$-dimensional de Rham cohomology of closed forms modulo exact forms. This is more than just as analogy: de Rham's theorem states that the de Rham cohomology of a smooth manifold is isomorphic to  $H^i\paren{M,\R}$ computed using a cubical complex structure on $M$~\cite{bredon2013topology}.

In some ways cohomology does not add more information than homology already provided. The universal coefficient theorem for cohomology tells us that $H_i\paren{X;\,\F} \simeq H^i\paren{X;\,\F}$ when $\F$ is a field (see Corollary 3.3 of~\cite{hatcher2002algebraic}). However, this algebraic equivalence does not preclude problems lending themselves more naturally to one perspective or the other. For example, it is arguably more intuitive to think of giant cycles as homological in nature, due to their interpretation as giant ``hypersurfaces'' spanning the surface. On the other hand, we argue that lattice gauge theory should be thought of as a random cochain and Wilson loop variables as cohomological quantities. It is also worth noting that although the homology and cohomology groups themselves may be isomorphic, there is a relationship between the different dimensional cohomology groups in the form of a multiplication of cocycles called the cup product that does not always have a homological analogue (see Section 3.2 of~\cite{hatcher2002algebraic}). 

Homology and cohomology come together in global duality theorems, of which there are many versions. We primarily use variations of Alexander duality, the original form of which says that for a sufficiently ``nice'' subspace $X \subset S^d,$
\[\Tilde{H}_i\paren{X;\,\F} \simeq \tilde{H}^{d-i-1}\paren{S^d\setminus{X};\,\F}\]
for $1\leq i \leq d-1.$ In the simplest case, this says that the number of bounded components of a ``nice'' subset of the plane equals the number of ``minimal contours'' of its complement. Since the complement of a set of plaquettes is a thickening of the dual system (as shown in Proposition~\ref{prop:deformationretract} below), we are able to use techniques related to the proof of Alexander duality to relate both the giant cycles and local cycles of the two sets (Theorem~\ref{thm:alexander} below). This relationship is quantified in terms of the dimension of the relevant subspaces of the homology groups, and is a main ingredient in the proof that the dual complex is close to a random-cluster model.

We will now briefly discuss the relevant information about the topology of the torus specifically. Although in percolation theory the torus is most often thought of as a cube with periodic boundary conditions, it is useful topologically to think of it as the product of $d$ copies of the circle $S^1.$ In such a product space, the K\"{u}nneth formula for homology (Section 3B of~\cite{hatcher2002algebraic}) tells us that there is an isomorphism of the form
\[\bigoplus_{i+j=k} H_i\paren{X;\,\F} \times H_j\paren{Y;\,\F} \simeq H_k\paren{X \times Y;\,\F}\,.\]
Since $H_1\paren{S^1;\,\F} \simeq H_0\paren{S^1;\,\F}\simeq \F$ and $H_i\paren{S^1;\,\F} = 0$ for all $i \geq 2,$ we see that 
\[H_i\paren{\T;\,\F} \simeq \F^{\binom{d}{i}}\,.\]
Furthermore, there is a set of generating $i$-cycles consisting of the products of $i$ of the $d$ possible $S^1$ factors.

\section{Topology and Duality}
\label{sec:topdual}

Before defining the plaquette random-cluster model, we will prove some useful duality results that are true for any cubical complex on the torus. Some of these results are from our previous paper on the Bernoulli plaquette model. The dual of the cubical complex $\T^d_N$ is the complex $\paren{\T^d_N}^{\bullet}$ obtained by shifting by $1/2$ in each coordinate direction, and the dual $\paren{\Z^d}^\bullet$ of $\Z^d$ is defined similarly. These pairs of dual complexes have the property that each $i$-plaquette intersects exactly one $(d-i)$-plaquette of $\paren{\T^d_N}^{\bullet},$ giving a matching of plaquettes and dual plaquettes. As in classical bond percolation in the plane, this induces a matching between subcomplexes of each. Namely, given a cubical complex $P = P\paren{\omega},$ we define the dual complex $P^{\bullet}$ to be the union of all plaquettes for which the dual plaquette is not included in $\omega.$ When $P$ comes from the plaquette random-cluster model defined below, it will include the union of all lower dimensional plaquettes of $\T^d_N$ (called the ``$(i-1)$-skeleton'') and no plaquettes of dimension higher than $i.$ Note that in that case, $P^{\bullet}$ automatically contains the  $(d-i-1)$-skeleton of $\paren{\T^d_N}^{\bullet}.$ Note that each statement in this section has an analogue for subcomplexes of a box in $\Z^d$ with free boundary, where the dual complex is a subcomplex of a box with wired boundary.

As mentioned in our review of homology and cohomology, our use of topological duality relies on decomposing the torus into a union of disjoint subspaces. The plaquette system and its its dual are not complementary, but the dual system deformation retracts (collapses) onto the complement of the plaquette system. By standard results in algebraic topology, this implies that the dual has the same topological invariants as the complement (see Chapter 0 of~\cite{hatcher2002algebraic}). 

\begin{Proposition}[Duncan, Kahle, and Schweinhart]\label{prop:deformationretract}
For any cubical complex $P \subseteq \T^d_N,$ $\T^d\setminus P$ deformation retracts to $P^{\bullet}.$ 
\end{Proposition}

Duality allows us to relates the homology of  $P$ with that of $P^{\bullet}.$ Importantly, ``global'' cycles and ``local'' cycles behave differently. Let $\phi_*:H_k\paren{P;\,\F}\rightarrow H_k\paren{\T^d;\,\F}$ and  $\psi_*:H_k\paren{P^{\bullet};\,\F}\rightarrow H_k\paren{\T^d;\,\F}$ be the maps on homology induced by the respective inclusions. Set $a_k = \dim \ker \phi_{k*}, a_k^\bullet = \dim \ker \psi_{k*}, b_k = \rank \phi_{k*},$ and $b_k^\bullet = \rank \psi_{k*}.$ We think of $a_k$ as counting the \emph{local cycles} (cycles of the plaquette system that are boundaries in the ambient torus) and $b_k$ as counting the \emph{global cycles}. Furthermore, recall that $\mathbf{b}_k = \dim H_k\paren{P}$ and let $\mathbf{b}^{\bullet}_k = H_{k}\paren{P^{\bullet}}.$ 

The following is a special case of the principle of Alexander duality, which is that the topology of a ``nice'' subset of a manifold should be related to that of its complement. Traditionally, Alexander duality is stated for a subset of the sphere $S^{d}$ or of Euclidean space $\R^d$, where there are no global cycles. The proof of the following relations for the torus makes use of the same techniques in the usual proof for those cases. These techniques, which include long exact sequences and relative homology, are beyond the level of the brief introduction in Section~\ref{sec:hom}. We refer an interested reader to~\cite{hatcher2002algebraic} for background. 

\begin{Theorem}[Alexander Duality]\label{thm:alexander}
For any cubical complex $P\subseteq \T^d_N$ and any $1\leq k\leq d-1$ the following relations hold.
\begin{align}\label{eq:1}
a_k + b_k = \mathbf{b}_k,\quad a^{\bullet}_k + b^{\bullet}_k = \mathbf{b}^{\bullet}_k\,,
\end{align}
\begin{align}
\label{eq:2}
b_k+b_k^\bullet=\rank H_k\paren{\T^d}=\binom{k}{i}\,,
\end{align}
and,
\begin{align}\label{eq:3}
a_k = a_{d-k-1}^\bullet\,.
\end{align}
\end{Theorem}

\begin{proof}
The first two equations are an immediate consequence of the rank--nullity theorem, and the third is Lemma~10 of~\cite{duncan2020homological}. To prove the fourth, we require additional results from algebraic topology. Theorem 3.44 of~\cite{hatcher2002algebraic} gives the isomorphism

\begin{align*}
    H^{d-i-1}\paren{P} \cong H_{i+1}\paren{\T^d_N,\T^d_N \setminus P}\,,
\end{align*}
where $H_i\paren{X,Y}$ denotes the relative homology group of $X$ with respect to a subset $Y.$ Combining this with the long exact sequence of relative homology (see Section 2.1 of~\cite{hatcher2002algebraic}), we obtain the following commutative diagram:

\[\begin{tikzcd}[cramped, sep = small]
H_{i+1}\paren{\T^d_N} \arrow[swap]{d}{\cong} \arrow{r}{\varphi} & H_{i+1}\paren{\T^d_N,\T^d_N \setminus P^{\bullet}} \arrow[swap]{d}{\cong} \arrow{r}{\chi} & H_{i}\paren{\T^d_N \setminus P^{\bullet}} \arrow{r}{\epsilon} & H_i\paren{\T^d_N}\\
H^{d-i-1}\paren{\T^d_N} \arrow{r}& H^{d-i-1}\paren{P^{\bullet}}
\end{tikzcd}
\,.\]

By Proposition~\ref{prop:deformationretract} and the definition of the plaquette system, 

\begin{align*}
    H_{i+1}\paren{\T^d_N \setminus P^{\bullet}} \cong H_{i+1}\paren{P} \cong 0\,,
\end{align*}

so $\varphi$ is surjective. Similarly,

\begin{align*}
    H_{i}\paren{\T^d_N \setminus P^{\bullet}} \cong H_{i}\paren{P}\,.
\end{align*}

Then since $\epsilon$ is the map on homology induced by the inclusion $\paren{\T^d_N \setminus P^{\bullet}} \hookrightarrow \T^d_N,$ its image is isomorphic to the space of giant cycles of $P.$ Thus, $\chi$ restricts to an isomorphism between vector spaces of dimension $\mathbf{b}^{\bullet}_{d-i-1} - b^{\bullet}_{d-i-1}$ and $\mathbf{b}_i - b_i$ respectively, and Equation~\ref{eq:3} follows from Equation~\ref{eq:1}. 
\end{proof}

At first glance, the dimensions in Equation~\ref{eq:3} may seem off: the plaquette random-cluster model is weighted by the $i$-th Betti number and is dual to a $(d-i)$-dimensional model, but this equation relates its $i$-th Betti number to the $(d-i-1)$-st Betti number of its dual. However, the random-cluster model always contains all cells of dimension lower than $j<i$ (that is, the $(i-1)$-skeleton). The presence or absence of $i$-cells can only affect the Betti numbers in dimensions $i$ and $(i-1),$ so these are the only variable Betti numbers in the random-cluster model. This, together with the Euler--Poincar\'{e} formula implies the following proposition.
\begin{Proposition}
\label{prop:c6}
There is a constant  $\newconstant\label{const:5} = \oldconstant{const:5}\paren{N,i,d}$ so that for any $i$-dimensional subcomplex $P$ of $\T^d_N$ containing the full $(i-1)$-skeleton
\begin{align}\label{eq:5}
    \mathbf{b}_i -\mathbf{b}_{i-1} =  \eta\paren{P} + \oldconstant{const:5}\,,
\end{align} 
where $\eta\paren{P}$ denotes the number of $i$-cells of $P.$ 
\end{Proposition}

\begin{proof}
    By the  Euler--Poincar\'{e} formula
\begin{align}
\sum_{j=0}^d \paren{-1}^j \mathbf{b}_j = \sum_{j=0}^{i-1} \paren{-1}^j \abs{F_N^j} + \eta\paren{\omega}
\end{align}
and the desired statement follows from the observations that $\abs{F_N^j}$ only depends on $N$ and $j$ for $1 \leq j \leq i-1,$ and that $\mathbf{b}_j$ also only depends on $N$ and $j$ when $1 \leq j \leq i-2.$
\end{proof}

\section{The Plaquette Random-Cluster Model}
\label{sec:RCM}
In this section we will investigate the properties of the plaquette random-cluster model. For the reader's convenience, we recall that the $i$-dimensional plaquette random-cluster model has the distribution defined by
\begin{align*}
    \mu_{X}\paren{P} = \mu_{X,p,q,i}\paren{P} \coloneqq \frac{1}{Z}p^{\eta\paren{P}}\paren{1-p}^{\abs{X^i} - \eta\paren{P}}q^{\mathbf{b}_{i-1}\paren{P}}\,,
\end{align*}

where $Z=Z\paren{X,p,q,i}$ is a normalizing constant~\cite{hiraoka2016tutte}.

Similarly to classical percolation, the random complex can alternatively be viewed as a a function designating plaquettes as open or closed. Certain probabilistic arguments are easier to present in this way, so we will introduce notation for this perspective here as well. We call a function $\omega : X^i \to \set{0,1}$ a configuration, and we refer to the elements of $\omega^{-1}\set{1}$ as open plaquettes and the elements of $\omega^{-1}\set{0}$ as closed plaquettes of $\omega,$ respectively. We then define the associated complex $P = P\paren{\omega}$ to be the union of the $(i-1)$-skeleton of $X$ and the open plaquettes of $\omega.$ 

\subsection{Positive Association}

Perhaps the most important property of the random-cluster model is the fact that it satisfies the FKG inequality. Hiraoka and Shirai showed that the plaquette random-cluster model also has this property, which we will use extensively. Since their proof is short, we reproduce it here for completeness.

\begin{Theorem}[Hiraoka and Shirai]\label{thm:i-FKG}
Let $p\in(0,1),$ and $q\geq 1,$ and $X$ a finite cubical complex. Then $\mu_{X}$ satisfies the FKG lattice condition and is thus positively associated, meaning that for any events $E,F$ that are increasing with respect to $\omega,$
\begin{equation*}
    \mu_{X}\paren{E\cap F} \geq \mu_{X}\paren{E}\mu_{X}\paren{F}.
\end{equation*}
\end{Theorem}

\begin{proof}
The key observation is that for any topological spaces $A,B$ and any $k \in \Z_{\geq 0},$
\begin{equation}\label{eq:mayervietorisbetti}
    \mathbf{b}_k\paren{A \cap B} + \mathbf{b}_k\paren{A \cup B} \geq \mathbf{b}_k\paren{A} + \mathbf{b}_k\paren{B}\,.
\end{equation}
To see this, consider a portion of the Mayer-Vietoris exact sequence (refer to Section 2.2 of~\cite{hatcher2002algebraic}):
\[\begin{tikzcd}
H_k\paren{A \cap B} \arrow{r}{\varphi}  & H_k\paren{A} \oplus H_k\paren{B} \arrow{r}{\chi} & H_k\paren{A \cup B} \arrow{r}{\partial} & H_{k-1}\paren{A \cap B}
\end{tikzcd}
\,.\]
The first isomorphism theorem gives
\begin{align*}
    \mathbf{b}_k\paren{A \cap B} = \rank \varphi + \dim \ker \varphi\,,\\
    \mathbf{b}_k\paren{A} + \mathbf{b}_k\paren{B}  = \rank \chi + \dim \ker \chi\,,\\
    \mathbf{b}_k\paren{A \cup B} = \rank \partial + \dim \ker \partial\,.
\end{align*}
By exactness, $\rank \varphi = \dim \ker \chi$ and $\rank \chi = \dim \ker \partial,$ so putting everything together yields
\[\mathbf{b}_k\paren{A \cap B} + \mathbf{b}_k\paren{A \cup B} - \mathbf{b}_k\paren{A} - \mathbf{b}_k\paren{B} = \dim \ker \varphi + \rank \partial \geq 0\,.\]

Now since $\mu_{X}$ is strictly positive for $p\in(0,1),$ it is enough to check the FKG lattice condition, which, written in terms of configurations, requires that 
\[\mu_{X}\paren{\omega \vee \omega'}\mu_{X}\paren{\omega \wedge \omega'} \geq \mu_{X}\paren{\omega} \mu_{X}\paren{\omega'}\]
for any configurations $\omega,\omega'.$ Since 
\[\eta\paren{\omega \vee \omega'} + \eta\paren{\omega \wedge \omega'} = \eta\paren{\omega} + \eta\paren{\omega'}\,,\]
we only need to check that
\[\mathbf{b}_{i-1}\paren{P \cap P'} + \mathbf{b}_{i-1}\paren{P \cup P'} \geq \mathbf{b}_{i-1}\paren{P} + \mathbf{b}_{i-1}\paren{P'}\,,\]
where $P'$ is the complex associated to $\omega'$ (where we are using the definition of the plaquette random-cluster model). This inequality is a special case of Equation~\ref{eq:mayervietorisbetti}.

\end{proof}

\subsection{The Plaquette Random-Cluster Model in Infinite Volume}
\label{sec:RCM_infinitevolume}

In this section we will show that, like the classical random-cluster model, the plaquette random-cluster model can be extended to an infinite volume setting. Specifically, the plaquette random-cluster model on $\Z^d$ will be defined via limits of plaquette random-cluster models on finite boxes $\Lambda_n \coloneqq [-n,n]^d \cap \Z^d.$ The proof given in Chapter 4 of~\cite{grimmett2006random} only requires minor changes to extend to higher dimensions, which are described here.

We will need to consider both free and wired boundary conditions for the plaquette random-cluster model in a subspace. A more general discussion of boundary conditions is given in Appendix~\ref{app:boundaryconditions}. The free boundary measure $\mu_{\Lambda_n}^{\mathbf{f}} = \mu_{\Lambda_n,p,q,i}^{\mathbf{f}}$ can be defined as the i-random-cluster model on the subcomplex with vertices $\Lambda_n.$ The wired boundary measure $\mu_{\Lambda_n}^{\mathbf{w}}=\mu_{\Lambda_n,p,q,i}^{\mathbf{w}}$ can be defined as the same measure conditioned on all plaquettes contained in $\partial [-n,n]^d$ being open. By an analogue of Theorem 4.19 of~\cite{grimmett2006random}, the weak limits of each of these measures exist.

\begin{Proposition}
Let $p \in \brac{0,1}$ and $q \geq 1.$
\begin{enumerate}[label=(\alph*)]
\item The limits $\mu^{\mathbf{f}}_{\Z^d} \coloneqq \lim_{n \to \infty} \mu_{\Lambda_n}^{\mathbf{f}}$ and $\mu_{\Z^d}^{\mathbf{w}} \coloneqq \lim_{n \to \infty} \mu_{\Lambda_n}^{\mathbf{w}}$ exist.
\item $\mu^{\mathbf{f}}_{\Z^d}$ and $\mu^{\mathbf{w}}_{\Z^d}$ are automorphism invariant.
\item $\mu^{\mathbf{f}}_{\Z^d}$ and $\mu^{\mathbf{w}}_{\Z^d}$ are positively associated.
\item $\mu^{\mathbf{f}}_{\Z^d}$ and $\mu^{\mathbf{w}}_{\Z^d}$ are tail-trivial, that is, for any event $A$ that is independent of the states of any finite set of plaquettes, $\mu^{\mathbf{f}}_{\Z^d}\paren{A}, \mu^{\mathbf{w}}_{\Z^d}\paren{A} \in \set{0,1}.$
\item For any nontrivial translation $\mathbf{t},$ $\mu^{\mathbf{f}}_{\Z^d}$ and $\mu^{\mathbf{w}}_{\Z^d}$ are $\mathbf{t}$-ergodic, that is, any $\mathbf{t}$-invariant random variable is $\mu^{\mathbf{f}}_{\Z^d}$-almost-surely and $\mu^{\mathbf{w}}_{\Z^d}$-almost-surely constant.
\end{enumerate}
\end{Proposition}

\begin{proof}
(a), (b), (d) Each proof is the same as in Theorem 4.19 of~\cite{grimmett2006random}

(c) Both measures are limits of positively associated measures by Lemma~\ref{lemma:positiveassociation}, so they are positively associated by Proposition 4.10 of~\cite{grimmett2006random}.

(e) The proof is the same as in Corollary 4.23 of~\cite{grimmett2006random}. 
\end{proof}

\subsection{Duality for the Random-Cluster Model}
In the classical random-cluster model in the plane, the dual graph is also distributed as a random-cluster model with parameters $q$ and $p^*$ where 
\[p^* \coloneqq \frac{\paren{1-p}q}{\paren{1-p}q + p}\,.\]
Beffara and Duminil-Copin~\cite{beffara2012self} used this relationship to prove that the critical probability is the self-dual point where $p = p^*,$ namely
\begin{align*}
    p_{\mathrm{sd}} \coloneqq \frac{\sqrt{q}}{1+\sqrt{q}}\,.
\end{align*}
In this section, we consider a similar duality in higher dimensions within the torus, showing that the dual of the $i$-dimensional plaquette random-cluster model is approximately a $(d-i)$-dimensional random-cluster model. We will see below that this recovers the duality properties of Potts lattice gauge theory as observed in~\cite{kogut1980z}, but in a way that perhaps provides more geometric intuition. Our proof makes heavy use of the topological duality results in Section~\ref{sec:topdual}, especially Theorem~\ref{thm:alexander}. We will use the notation from that theorem.

Since it will be important to distinguish between giant and local cycles, we recall their definitions. Let $\phi : P \hookrightarrow X$ be the inclusion map and let $\phi_* : H_i\paren{P} \to H_i\paren{X}$ be the induced map on $i$th homology. We say that an element $\alpha \in H_i\paren{P}$ is a giant cycle if $I_*\paren{\alpha} \neq 0$ and a local cycle if $i_*\paren{\alpha} = 0.$ Recall that $b_i\paren{P} = \rank \phi_*$ is informally the number of giant cycles in $P.$

First we define a \say{balanced} version~\cite{beffara2012self} which satisfies exact duality:

\begin{align*}
    \tilde{\mu}_{\T^d_N}\paren{P} \coloneqq \frac{\paren{\sqrt{q}}^{-b_i\paren{P}}}{\tilde{Z}}p^{\eta\paren{P}}\paren{1-p}^{\abs{F^i_N} - \eta\paren{P}}q^{\mathbf{b}_{i-1}\paren{P}}\,.
\end{align*}
The additional term $\paren{\sqrt{q}}^{-b_i\paren{P}}$ term ``corrects'' for the different behavior of local and global cycles under duality. The two models are absolutely continuous with respect to each other with a Radon-Nikodym derivatives bounded above and below by functions of $q$ and the same is true for their dual models. As such, a sharp threshold for $\tilde{\mu}_{\T^d_N}$ implies one for $ \tilde{\mu}_{\T^d_N}.$ 

\begin{Theorem}\label{theorem:duality}
The balanced plaquette random-cluster model satisfies
\begin{equation*}
    \tilde{\mu}_{\T^d_N,p,q,i}\paren{P} \buildrel d \over = \tilde{\mu}_{\T^d_N,p^*,q,d-i}\paren{P^\bullet}\,.
\end{equation*}
\end{Theorem}

\begin{proof}
The idea of the proof is the same as in the classical random-cluster model. We need only take care to keep track of the giant cycles and local cycles separately, since they behave differently under duality. Note that

\begin{align}\label{eq:4}
\eta\paren{P} + \eta\paren{P^{\bullet}} = \abs{F_N^i}\,.
\end{align}

Recall also that $P$ and $P^{\bullet}$ contain the complete $(i-1)$-skeleton and $(d-i-1)$-skeleton respectively, so

\begin{align}\label{eq:8}
    b_{i-1} = \binom{d}{i-1},\qquad b_{d-i-1} = \binom{d}{d-i-1}\,.
\end{align}

It is not crucial to the argument, but we can simplify the upcoming calculation slightly to note from the bijection between plaquettes and dual plaquettes that 
\begin{align}\label{eq:6}
    \abs{F^i_N} = \abs{F^{d-i}_N}\,.
\end{align}

Lastly, we recall the following property of $p^*,$
\begin{align}\label{eq:7}
    \frac{pp^*}{\paren{1-p}\paren{1-p^*}} = q \,.
\end{align}

For convenience, in the following calculation we will denote $C \coloneqq \abs{F^i_N} = \abs{F^{d-i}_N}$ (equal by Equation~\ref{eq:6}).
Now we compute  $\tilde{\mu}_{\T^d_N,p,q,i}\paren{P}=$
\begin{align*}
   &\phantom{=}\frac{\paren{\sqrt{q}}^{-b_i}}{\tilde{Z}}p^{\eta\paren{P}}\paren{1-p}^{\abs{F^i_N}-\eta\paren{P}}q^{\mathbf{b}_{i-1}}\\
    &= \frac{\paren{1-p}^{\abs{F^i_N}}}{\tilde{Z}}\paren{\sqrt{q}}^{-b_i}\paren{\frac{p}{1-p}}^{\eta\paren{P}}q^{\mathbf{b}_{i-1}}\\
    &= \frac{q^{\oldconstant{const:5}}\paren{1-p}^{\abs{F^i_N}}}{\tilde{Z}}\paren{\sqrt{q}}^{-b_i}\paren{\frac{p}{q\paren{1-p}}}^{\eta\paren{P}}q^{\mathbf{b}_{i}}\;&&\text{by~(\ref{eq:5})}\\
    &= \frac{q^{\oldconstant{const:5}}\paren{1-p}^{\abs{F^i_N}}}{\tilde{Z}}\paren{\sqrt{q}}^{-b_i}\paren{\frac{q(1-p)}{p}}^{-\eta\paren{P}}q^{a_i + b_i} && \text{by~(\ref{eq:1})}\\
    &= \frac{q^{\oldconstant{const:5}}\paren{1-p}^{\abs{F^i_N}}}{\tilde{Z}}\paren{\sqrt{q}}^{-b_i}\paren{\frac{p^*}{1-p^*}}^{-\eta\paren{P}}q^{a_{i} + b_{i}}&& \text{by~(\ref{eq:7})}\\
    &= \frac{q^{\oldconstant{const:5}+\binom{d}{i}/2}\paren{1-p}^{\abs{F^i_N}}}{\tilde{Z}}\paren{\sqrt{q}}^{b^{\bullet}_{d-i}}\paren{\frac{p^*}{1-p^*}}^{-\eta\paren{P}}q^{a_{d-i-1}^{\bullet} - b_{d-i}^{\bullet}}&& \text{by~(\ref{eq:2}),~(\ref{eq:3})}\\
    &= \frac{q^{\oldconstant{const:5}+\binom{d}{i}/2}\paren{1-p}^{\abs{F^i_N}}}{\tilde{Z}}\paren{\sqrt{q}}^{b^{\bullet}_{d-i}}\paren{\frac{p^*}{1-p^*}}^{\eta\paren{P^{\bullet}} - \abs{F^i_N}}q^{a_{d-i-1}^{\bullet} - b_{d-i}^{\bullet}}&& \text{by~(\ref{eq:4})}\\
    &= \frac{q^{\oldconstant{const:5}+\binom{d}{i}/2}\paren{1-p}^{\abs{F^i_N}}}{q^{b^{\bullet}_{d-i-1}}\tilde{Z}}\paren{\sqrt{q}}^{-b^{\bullet}_{d-i}}\paren{\frac{p^*}{1-p^*}}^{\eta\paren{P^{\bullet}} - \abs{F^i_N}}q^{\mathbf{b}_{d-i-1}^{\bullet}}&& \text{by~(\ref{eq:1}),~(\ref{eq:2})}\\
    &= \frac{q^{\oldconstant{const:5}+\binom{d}{i}/2}\paren{1-p}^{\abs{F^i_N}}}{q^{\binom{d}{d-i-1}}p^{\abs{F^{i}_N}}\tilde{Z}} \paren{\sqrt{q}}^{-b^{\bullet}_{d-i}} \frac{\paren{p^*}^{\eta\paren{P^{\bullet}}}}{\paren{1-p^*}^{\eta\paren{P^{\bullet}} - \abs{F^{d-i}_N}}}q^{\mathbf{b}_{d-i-1}^{\bullet}}&& \text{by~(\ref{eq:8})}\\
    &\coloneqq \frac{\paren{\sqrt{q}}^{-b^{\bullet}_{d-i}}}{\tilde{Z}^\bullet}\paren{p^*}^{\eta\paren{P^{\bullet}}}\paren{1-p^*}^{\abs{F^{d-i}_N} - \eta\paren{P^{\bullet}}}q^{\mathbf{b}_{d-i-1}^{\bullet}}\\ 
    &= \tilde{\mu}_{\T^d_N,p^*,q,d-i}\paren{P^{\bullet}}\,,
\end{align*}
where we used the definition of the $(d-i)$-dimensional random-cluster model in the last step.
\end{proof}

As a corollary, we have the following duality relationship between the normalizing constants.
\begin{Corollary}
\label{corollary:normalizing}
\[Z(p^*,q,d-i,N)=q^{c+\binom{d}{i}/2-\binom{d}{d-i-1}}\paren{1-p}^{\abs{F^i_N}}Z(p,q,i,N)\,,\]
where $c$ is the constant $\oldconstant{const:5}$ defined in Proposition~\ref{prop:c6}.
\end{Corollary}

\section{Relation to Potts Lattice Gauge Theory}
\label{sec:latticegauge}
We now elucidate the relationship between the $i$-dimensional plaquette random-cluster and $(i-1)$-dimensional Potts lattice gauge theory. In particular, we find a topological interpretation of the generalized Wilson loop variables. First, we review the coupling between these models, which generalizes the Edwards--Sokal coupling between the classical random-cluster model and the Potts model~\cite{fortuin1972random,edwards1988generalization}.

\begin{Proposition}[Hiraoka and Shirai~\cite{hiraoka2016tutte}]\label{prop:coupling}
Let $X$ be a finite cubical complex, $q$ be a prime, $p \in [0,1),$ and $p = 1-e^{-\beta}.$ Consider the coupling on $C^{i-1}\paren{X}\times \set{0,1}^{X^i}$ defined by
\[\kappa\paren{f,P\paren{\omega}} \propto \prod_{\sigma \in X^i}\brac{\paren{1-p}K\paren{\omega\paren{\sigma},0} + p K\paren{\omega\paren{\sigma},1}K\paren{\delta f\paren{\sigma},0}}\,,\]
where $K(x,y)$ is the Kronecker delta function. Then $\kappa$ has the following marginals.
\begin{itemize}
    \item The first marginal is the $q$-state Potts lattice gauge theory with inverse temperature $\beta$ given by
    \[\sum_{P} \kappa\paren{f,P} \propto e^{-\beta H(f)}\,,\]
    where $H(f)$ is the Hamiltonian for Potts lattice gauge theory:
        \[H\paren{f}=\sum_{\sigma}K\paren{\delta f\paren{\sigma},1}\,.\]
    \item The second marginal is the plaquette random-cluster model with parameters $p,q$ given by
    \[\sum_{f \in C^{i-1}\paren{X}} \kappa\paren{f,P} \propto p^{\eta\paren{P}}\paren{1-p}^{\abs{X^i} - \eta\paren{P}}q^{\mathbf{b}_{i-1}\paren{P;\,\F_q}}\,.\]
\end{itemize}
\end{Proposition}

\begin{proof}
We include Hiraoka and Shirai's proof adapted to our notation for completeness. The first marginal is calculated as
\begin{align*}
    \kappa_1\paren{f} & \coloneqq \sum_{\omega \in \set{0,1}^{X^i}} \kappa\paren{f,P\paren{\omega}}\\ &\propto \sum_{\omega \in \set{0,1}^{X^i}} \prod_{\sigma \in X^i} \brac{\paren{1-p}K\paren{\omega\paren{\sigma},0}+p K\paren{\omega\paren{\sigma},1}K\paren{\delta f\paren{\sigma},0}}\\
    &= \prod_{\sigma \in X^i} \brac{\paren{1-p}+p K\paren{\delta f\paren{\sigma},0}}\\
    &= \prod_{\sigma \in X^i} \brac{e^{-\beta}+\paren{1-e^{-\beta}} K\paren{\delta f\paren{\sigma},0}}\\
    &= e^{-\beta\abs{X^i}}\prod_{\sigma \in X^i} \brac{1+\paren{e^{\beta}-1} K\paren{\delta f\paren{\sigma},0}}\\
    &= e^{-\beta\abs{X^i}}e^{-\beta H(f)} \\
    &\propto e^{-\beta H(f)}\,.
\end{align*}

The second marginal is calculated as
\begin{align*}
    \kappa_2\paren{P} & \coloneqq \sum_{f \in C^{i-1}\paren{X}} \kappa\paren{f,P\paren{\omega}}\\
    &\propto \sum_{f \in C^{i-1}\paren{X}} \prod_{\sigma \in X^i} \brac{\paren{1-p}K\paren{\omega\paren{\sigma},0}+p K\paren{\omega\paren{\sigma},1}K\paren{\delta f\paren{\sigma},0}}\\
    &= \paren{1-p}^{\abs{X^i} -\eta\paren{P}}p^{\eta\paren{P}}\sum_{f \in C^{i-1}\paren{X}} \prod_{\substack{\sigma \in X^i\\\omega\paren{\sigma}=1}} K\paren{\delta f\paren{\sigma},0}\\
    &= \paren{1-p}^{\abs{X^i} -\eta\paren{P}}p^{\eta\paren{P}}q^{\dim Z^{i-1}\paren{P;\,\F_q}}\\
    &\propto \paren{1-p}^{\abs{X^i} -\eta\paren{P}}p^{\eta\paren{P}}q^{\mathbf{b}_{i-1}\paren{P;\,\F_q}}\,.
\end{align*}
The second to last line follows from the definition of a cocycle. The last line holds because the $(i-1)$-skeleton of $P$ does not depend on $\omega,$ so $B^{i-1}\paren{P;\,\F_q}$ is fixed.
\end{proof}

Next, we compute the conditional measures defined by the coupling.

\begin{Proposition}\label{prop:conditionalmeasures}
Let $p = 1-e^{-\beta}.$ Then the conditional measures of $\kappa$ are as follows.
\begin{itemize}
    \item Given $f,$ the conditional measure $\kappa\paren{\cdot \mid f}$ is Bernoulli plaquette percolation with probability $p$ on the set of plaquettes $\sigma$ that satisfy $\delta f\paren{\sigma} = 0.$
    \item Given $P,$ the conditional measure $\kappa\paren{\cdot \mid P}$ is the uniform measure on $\paren{i-1}$-cocycles in $Z^{i-1}\paren{P;\,\F_q}.$
\end{itemize}
\end{Proposition}

\begin{proof}
From the definition of the coupling, under $\kappa\paren{\cdot \mid s}$ a plaquette $\sigma$ is open with probability $p$ independently of other plaquettes when $\delta f\paren{\sigma} = 0$ and always closed otherwise, giving the first conditional measure. The second conditional measure is determined by the observation that $\kappa\paren{\cdot \mid P}$ is supported on the set of $f \in C^{i-1}\paren{X}$ satisfying $\delta f\paren{\sigma} = 0$ for each $\sigma \in P,$ and that each such cochain has the same weight.
\end{proof}

In the next section, we discuss how to use this proposition to sample Potts lattice gauge theory.

\subsection{The Plaquette Swendsen--Wang Algorithm}
\label{sec:sw}
The classical Swendsen--Wang algorithm was the first non-local Monte Carlo algorithm for the Potts model, which alternately samples from the marginals in the combined random-cluster representation~\cite{swendsen1987nonuniversal,edwards1988generalization}. That is, given a coupling constant $\beta,$ a spin configuration $f\in C^0\paren{X;\,\F_q}$ is updated by first sampling a random graph $G$ by performing percolation with probability $p=1-1^{-\beta}$ on the edges between vertices with equal spins. Then, the spins are resampled uniformly on each component of $G.$ In our language, this corresponds to sampling a uniformly random cocycle in $Z^0\paren{G;\,\F_q}.$ The Swendsen--Wang algorithm is observed to converge significantly faster than Glauber dynamics, especially at and near criticality~\cite{swendsen1987nonuniversal}.   

We generalize the Swendsen--Wang algorithm to sample Potts lattice gauge theory using the coupled representation in Proposition~\ref{prop:coupling}. That is, given a cochain  $f\in C^{i-1}\paren{X;\,\F_q},$ a random $i$-complex $P$ is sampled by including plaquettes on which $\delta f$ vanishes independently with probability $p=1-e^{-\beta}.$ An updated cochain is then found by selecting a uniformly random cocycle from $Z^{i-1}\paren{P;\,\F_q}.$ It is easily seen that the resulting Markov chain is ergodic and satisfies detailed balance, with the stationary distribution being Potts lattice gauge theory. An implementation of this algorithm will be described in~\cite{PlaquetteSW}. 

When $d=2i$ and $q$ is an odd prime, we can interpret Theorem~\ref{thm:half} in terms of the qualitative behavior of the Swendsen--Wang algorithm on the torus at the self-dual point. Below the threshold, the behavior of the algorithm is ``local'' in the sense that all co-cycles in $Z^i\paren{P;\,\F_q}$ are trivial in the homology of the torus. On the other hand, when $p>p_c,$  $Z^i\paren{P;\,\F_q}$ contains $\binom{d}{i}$ giant cocycles up to cohomology (each given the same weight), and a ``non--local'' move happens with probability approaching $1-q^{-\binom{d}{i}}$ (as $P$ has $\binom{d}{i}$ giant cocycles, with high probability)\footnote{While we state our theorems in terms of giant cycles, the analogous results for giant cocycles follow immediately from them, either by the functoriality of $\mathrm{Hom}\paren{\cdot,\F_q}$ or by repeating the same proofs.}. Interestingly, at $p=p_c,$ the number of giant cocycles is somewhere between $0$ and $\binom{d}{i}$ with probability bounded away from $0$ and $1,$ so the probability of non-local behavior is non-zero but lower than above $p_c.$ Theorems~\ref{thm:one} and~\ref{thm:weak} have a similar interpretation.

\subsection{Potts Lattice Gauge Theory in Infinite Volume}
\label{sec:potts_infinitevolume}
We can use the infinite volume random-cluster measure to define the corresponding infinite volume Potts lattice gauge theory. We start with finite volume measure $\nu^{\mathbf{f}}_{\Lambda_n} =\nu^{\mathbf{f}}_{\Lambda_n,\beta,q,i-1,d},$ which is defined to be the usual Potts lattice gauge theory on the subcomplex induced by the vertices of $\Lambda_n$ (we use the superscript $\mathbf{f}$ because this measure can be coupled with the free random-cluster measure $\mu_{\Lambda_n}^{\mathbf{f}}$ as in Proposition~\ref{prop:coupling}). Next, we construct the limiting measure $\nu^{\mathbf{f}} = \nu^{\mathbf{f}}_{\beta,q,i-1,d}$ by taking $n \to \infty.$ A similar construction can be done with wired boundary conditions, but we will omit the details here. We start with an easy lemma that will allow us to choose random cocycles in increasing subcomplexes of $P$ in a consistent manner. 

\begin{Lemma}\label{lemma:finitesupport}
For any configuration $\omega \in \Omega,$ there is a basis $\mathcal{B}$ of $Z^{i-1}\paren{P\paren{\omega},F}$ so that each $(i-1)$-face of $\Z^d$ is in the support of finitely many elements of $\mathcal{B}.$
\end{Lemma}

\begin{proof}
Take any basis and perform Gaussian elimination.
\end{proof}

We will call the minimal such basis with respect to lexicographical order $\mathcal{B}_{\omega}$
\begin{Corollary}
    Let $p \in [0,1), q \in \set{2,3,\ldots},$ and $p = 1-e^{-\beta}.$
    \begin{enumerate}[label=(\alph*)]
        \item Let $\omega$ be distributed according to $\mu^{\mathbf{f}}_{\Z^d}.$ Conditional on $\omega,$ let $\set{A_g : g \in \mathcal{B}_{\omega}}$ be i.i.d. $\mathrm{Unif}\paren{\F_q}$ random variables. Then the limit 
        \[\nu^{\mathbf{f}} \coloneqq \lim_{n \to \infty} \nu^{\mathbf{f}}_{\Lambda_n}\]
        exists, and the random cocycle 
        \[f = \sum_{g \in \mathcal{B}} A_g g\]
        is distributed according to $\nu^{\mathbf{f}}.$
        \item Let $f$ be distributed according to $\nu^{\mathbf{f}}.$ Conditional on $f,$ let $\omega$ be a random configuration in which each $i$-plaquette $\sigma$ is open with probability $p$ if $f\paren{\partial \sigma} = 0$ independent of the states of other plaquettes and closed otherwise. Then $\omega$ is distributed according to $\mu^{\mathbf{f}}_{\Z^d}.$
\end{enumerate}
\end{Corollary}

\begin{proof}
We proceed similarly to proof of Theorem 4.91 of~\cite{grimmett2006random}.

(a) By the proof of Theorem 4.19 of~\cite{grimmett2006random}, there is an increasing set of configurations $\omega_n$ so that each $\omega_n$ is distributed according to $\mu_{\Lambda_n}^{\mathbf{f}}$ and $\lim_{n \to \infty} \omega_n$ is distributed according to $\mu_{\Z^d}^{\mathbf{f}}.$ Moreover, for any $i$-plaquette $\sigma,$ $\omega_n\paren{\sigma} = \omega\paren{\sigma}$ for large enough $n.$

Now let $\set{A_g : g \in C^{i-1}\paren{\Z^d;\,\F_q}}$ be i.i.d. $\mathrm{Unif}\paren{\F_q}$ random variables, and let 
\[f_n \coloneqq \sum_{g \in \mathcal{B}_{\omega_n}} A_g g\,.\]
By the construction of $\mathcal{B}_{\omega_n},$ $f_n$ is distributed according to $\nu^{\mathbf{f}}$ and is eventually constant on any finite set of $(i-1)$-cubes. Since $\mathcal{B}_{\omega_n} \to \mathcal{B}_{\omega},$ $\lim_{n \to \infty} f_n$ is distributed as $\nu^{\mathbf{f}}$ and we are done.

(b) Let $\set{B_{\sigma} : \sigma \in F^i}$ be independent $\mathrm{Ber}\paren{p}$ random variables. Also, let $f_n$ be distributed as $\nu^{\mathbf{f}}_{\Lambda_n},$ coupled as before so that $f_n$ is eventually constant on any finite set of $(i-1)$-cubes. Then, Proposition~\ref{prop:conditionalmeasures}, $\omega_n\paren{\sigma}\coloneqq B_{\sigma} K\paren{f_n\paren{\partial \sigma},1}$ is distributed according to $\mu_{\Lambda_n}^{\mathbf{f}}.$ Therefore $\omega \coloneqq \lim_{n \to \infty}\omega_n$ is distributed according to $\mu_{\Z^d}^{\mathbf{f}} = \lim_{n \to \infty} \mu_{\Lambda_n}^{\mathbf{f}}.$
\end{proof}

In the remainder of this section, we will refer to ``Potts lattice gauge theory on $\Z^d$'' and ``the plaquette random-cluster model on $\Z^d$'' without explicit reference to the boundary conditions used to construct the infinite volume limit. Note that the constants defined below may depend on this choice.

\subsection{Duality for Potts Lattice Gauge Theory}
Another corollary of the coupling between Potts lattice gauge theory and the plaquette random-cluster model is a new proof of duality for the partition functions of Potts lattice gauge theory. This duality was first observed in the paper that introduced Potts lattice gauge theory~\cite{kogut1980z}. 

\begin{Corollary}
Let $\mathcal{Z}\paren{\T^d_N,q,\beta,i}$ be the normalizing constant for $q$-state $i-1$-dimensional Potts lattice gauge theory on $\T^d_N.$ Then, there are constants  $0<\newconstant\label{const:Z1},\newconstant\label{const:Z2}<\infty$ so that   
\[\oldconstant{const:Z1} \mathcal{Z}\paren{\T^d_N,q,\beta,i} \leq \mathcal{Z}\paren{\T^d_N,q,\beta^*,d-i} \leq \oldconstant{const:Z2}  \mathcal{Z}\paren{\T^d_N,q,\beta,i}\]
where
\[\beta^*=\beta^*\paren{\beta,q}=\log\paren{\frac{e^{\beta}+q-1}{e^{\beta}-1}}\,.\]
\end{Corollary}
\begin{proof}
This follows from Corollary~\ref{corollary:normalizing}, the relationship between the balanced and unbalanced plaquette random cluster models on the torus, and the computations in the proof of Proposition~\ref{prop:coupling}. 
\end{proof}
Note that the proofs can be adapted to relate the partition function for $i-1$-dimensional  Potts lattice gauge theory on a box in $\Z^d$ with free boundary conditions to that of  $(d-i-1)$-dimensional  Potts lattice gauge theory on a box in $\Z^d$ with wired boundary conditions.

\subsection{Generalized Wilson Loop Variables and Proof of Theorem~\ref{thm:comparison}}

In this section, let $X$ either be a finite cubical complex or $\Z^d.$ Let $\gamma\in Z_{i-1}\paren{X,F}$ be an $(i-1)$-dimensional cycle. It is natural to ask whether $\gamma$ is a boundary in the $i$-dimensional plaquette random-cluster model.

\begin{Definition}
Let $V_{\gamma}$ to be the event that $\gamma$ is null-homologous in $H_i\paren{P;\,\F}$ for a specified coefficient field $\F.$
\end{Definition}

The generalized Wilson loop variable is a closely related quantity for Potts lattice gauge theory.

\begin{Definition}
Let $f$ be an $(i-1)$-cocycle. Recall that the \emph{generalized Wilson loop variable} associated to $\gamma$ is the random variable $W_{\gamma}:C^{i-1}\paren{X;\,\F_q}\rightarrow \C$ given by 
\[W_{\gamma}\paren{f}=\paren{f\paren{\gamma}}^{\mathbb{C}}\,.\]
\end{Definition}

We are particularly interested in the behavior of these random variables when $\gamma$ is an $(i-1)$-dimensional hyperrectangular boundary in $\Z^d$ --- that is, the boundary of a hyperrectangle  $\brac{0,n_1}\times \ldots \times \brac{0,n_{i}} \times \mathbf{x}$ inside $\R^{i}$ where $n_j\in \Z$ for $j=1,\ldots,i$ and $\mathbf{x}\in \Z^{d-i},$  or a congruent object obtained translation and/or by permuting the coordinates of $\Z^d.$ The next several results show a close connection between Wilson loop variables and the topology of the plaquette random-cluster model. 

We first show that Wilson loop variables are perfectly correlated if they are homologous, and are otherwise independent. For completeness, we begin by proving an elementary topological lemma which will be useful in a couple places.
\begin{Lemma}
\label{lemma:constclass}
Let $\gamma \in Z_{i-1}\paren{X;\,\F_q}$ and  $f \in Z^{i-1}\paren{X;\,\F_q}.$ Then for any $f' \in Z^{i-1}\paren{X;\,\F_q}$ so that $\brac{f} = \brac{f'}$ in $H^{i-1}\paren{X;\,\F_q}$ and any $\gamma' \in Z_{i-1}\paren{X;\,\F_q}$ so that $\brac{\gamma} = \brac{\gamma'}$ in $H_{i-1}\paren{X;\,\F_q},$ we have $f\paren{\gamma} = f'\paren{\gamma'}.$
\end{Lemma}

\begin{proof}
First we show that $f\paren{\gamma} = f'\paren{\gamma}.$ If $i=1,$ $f=f'$ and there is nothing to prove. Now, assume $i \geq 2.$ For $h \in F^{i-2}_N,$ let $h^*\in C^{i-2}\paren{X;\,\F_q}$ be the dual of $h$ when $h$ is viewed as an element of  $C_{i-2}\paren{X;\,\F_q}.$  Note that since $\gamma$ is a cycle, 
\[\delta h^*\paren{\gamma} = h^*\paren{\partial \gamma}= 0\,.\]
Then, because $f$ and $f'$ are cohomologous,  we can write 
\[f-f' = \sum_{h \in X^{i-2}} c_h \delta h^*\]
for some $\set{c_h} \in \F_q^{X^{i-2}}.$ It follows that
\[f\paren{\gamma} = f'\paren{\gamma}  + \sum_{h \in X^{i-2}} c_h \delta h^*\paren{\gamma} = f'\paren{\gamma}\,.\]

Now as $\brac{\gamma} = \brac{\gamma'},$ there is a chain $\rho \in C_i\paren{X;\,\F_q}$ so that $\gamma = \partial \rho + \gamma'.$ Then since $f'$ is a cocycle, we can calculate
\begin{align*}
    f'\paren{\gamma} &= f'\paren{\partial \rho} + f'\paren{\gamma'}\\
    &= \delta f'\paren{\rho} + f'\paren{\gamma'}\\
    &= f'\paren{\gamma'}\,.
\end{align*}
\end{proof}

We now describe the distribution of Wilson loops as a function of the associated plaquette random-cluster model.
\begin{Proposition}
\label{prop:wilsondist}
Let $q$ be a prime integer and let $\gamma \in Z_{i-1}\paren{X;\,\F_q}.$ Then 
\[\paren{W_\gamma \mid V_{\gamma}} \equiv  1\]
and
\[\paren{W_\gamma \mid V_{\gamma}^c} \sim  \mathrm{Unif}\paren{\F_q}^{\mathbb{C}}\,.\]
\end{Proposition}

\begin{proof}
Fix $\omega$ and $\gamma.$ If $0 = \brac{\gamma} \in H_{i-1}\paren{P;\,\F_q},$ then $f\paren{\gamma} = 0$ by Lemma~\ref{lemma:constclass}.

Now assume that $0 \neq \brac{\gamma} \in H_{i-1}\paren{P;\,\F_q}.$ By Proposition~\ref{prop:conditionalmeasures}, $\nu\paren{f \mid \omega}$ can be sampled by fixing a basis of $H^{i-1}\paren{P}$ and taking a random linear combination with independent $\mathrm{Unif}\paren{\F_q}$ coefficients. Thus, $f\paren{\gamma}$ is uniformly distributed on an additive subgroup of $\F_q.$ Since the only such subgroups are $\F_q$ and $\set{0},$ we only need to rule out the latter. By the universal coefficient theorem for cohomology (Corollary 3.3 of~\cite{hatcher2002algebraic}), there is a dual element $0 \neq \brac{\gamma}^* \in H^{i-1}\paren{P;\,\F_q}$ so that $\brac{\gamma}^*\paren{\brac{\gamma}} \neq 0.$ Therefore, $f\paren{\gamma}$ is distributed as $\mathrm{Unif}\paren{\F_q},$ so $W_{\gamma}$ is distributed as $\mathrm{Unif}\paren{\F_q}^{\mathbb{C}}.$
\end{proof}

The proposition has a few interesting corollaries.

\begin{Corollary}
\label{cor:VW}
Let $\gamma \in Z_{i-1}\paren{X;\,\F_q}.$
\[\mathbb{E}_{\nu}\paren{W_\gamma \mid V_\gamma} = 1\]
and
\[\mathbb{E}_{\nu}\paren{W_\gamma \mid V_\gamma^c} = 0\,.\]
In particular, if $V_\gamma$ is the event that $\gamma$ is null-homologous in $P$ then
\[\mathbb{E}_{\nu}\paren{W_{\gamma}}=\mu_{X}\paren{V_{\gamma}}\,.\]
\end{Corollary}
The last statement is Theorem~\ref{thm:comparison}.

\begin{Corollary}
Let $\gamma,\gamma' \in Z_{i-1}\paren{X;\,\F_q}.$ Then conditioned on 
\[\brac{\gamma} = \brac{\gamma'} \in H_{i-1}\paren{P;\,\F_q}\,,\] $W_{\gamma} = W_{\gamma'}$ almost surely. Conversely, conditioned on 
\[\brac{\gamma} \notin \mathrm{span}\paren{\brac{\gamma'}} \subset H_{i-1}\paren{P;\,\F_q}\,,\]
$W_{\gamma}$ and $W_{\gamma'}$ are independent.
\end{Corollary}
\begin{proof}
Conditioned on $\brac{\gamma} = \brac{\gamma'},$ we can simply apply Proposition~\ref{prop:wilsondist} to the Wilson loop variable for $\gamma-\gamma'$ to get the desired result.

Now condition on the event $D_{\gamma,\gamma'} \coloneqq \set{\brac{\gamma} \notin \mathrm{span}\paren{\brac{\gamma'}}}.$ First, we show independence of $W_{\gamma}$ and $W_{\gamma'}$ when further conditioning on a fixed configuration $\xi.$ If $\brac{\gamma'} = 0 \in H_{i-1}\paren{P\paren{\xi},\F_q},$ then $W_{\omega}$ and $W_{\omega'}$ are trivially independent. Now assume $\brac{\gamma'} \neq 0.$ Since $D_{\gamma,\gamma'}$ holds, we can extend the set $\set{\brac{\gamma}^*,\brac{\gamma'}^*}$ to a basis $\mathcal{B}$ of $H^{i-1}\paren{P_{\xi},\F_q}$ again using the universal coefficent theorem for cohomology. Then by Proposition~\ref{prop:conditionalmeasures}, we can sample $\nu\paren{f \mid \omega = \xi}$ by taking a random linear combination of elements of $\mathcal{B}$ with independent $\mathrm{Unif}\paren{\F_q}$ coefficients. By construction, we then see that $f\paren{\omega}$ and $f\paren{\omega'}$ are independent, and thus $W_{\gamma}$ and $W_{\gamma'}$ are independent.

Now note that $\paren{W_{\gamma} \mid D_{\gamma,\gamma'}, \omega = \xi} \sim \mathrm{Unif}\paren{\F_q}^{\mathbb{C}}$ by the same argument as in Proposition~\ref{prop:wilsondist}. Then for any $\zeta_1,\zeta_2 \in \paren{\F_q}^{\mathbb{C}},$ we can compute
\begin{align*}
    &\nu\paren{W_{\gamma} = \zeta_1 \mid W_{\gamma'} = \zeta_2, D_{\gamma,\gamma'}} \\
     &\quad = \sum_{\xi} \nu\paren{W_{\gamma} = \zeta_1 \mid  \omega = \xi, W_{\gamma'} = \zeta_2, D_{\gamma,\gamma'}} \nu\paren{\omega = \xi \mid W_{\gamma'} = \zeta_2, D_{\gamma,\gamma'}}\\
    &\quad= \sum_{\xi} \mathbb{P}\paren{\mathrm{Unif}\paren{\F_q}^{\mathbb{C}} = \zeta} \nu\paren{\omega = \xi \mid W_{\gamma'} = \zeta_2, D_{\gamma,\gamma'}}\\
    &\quad= \mathbb{P}\paren{\mathrm{Unif}\paren{\F_q}^{\mathbb{C}} = \zeta_1}\\
    &\quad= \nu\paren{W_{\gamma} = \zeta_1 \mid D_{\gamma,\gamma'}}\,.
\end{align*}
Thus, $W_{\gamma}$ and $W_{\gamma'}$ are independent conditioned on $D_{\gamma,\gamma'}.$

\end{proof}

Finally, we prove a direct generalization of the relationship between the two-point correlation function of the classical Potts model and connection probabilities in the corresponding random-cluster model (Theorem 1.16 in~\cite{grimmett2006random}). For any $\gamma\in Z_{i-1}\paren{X,\F_q},$ let
\[\tau_{\beta,q}\paren{\gamma}=\nu\paren{W_{\gamma}=1}-\frac{1}{q}\,.\]
Note that if $i=1$ and $v_0$ and $v_1$ are sites of $X,$ then  $\tau_{\beta,q}\paren{v_1-v_0}$ is the two-point correlation between the sites.

\begin{Proposition}
\[\tau_{\beta,q}\paren{\gamma}=\paren{1-\frac{1}{q}}\mu_{X}\paren{V_{\gamma}}\,,\]
where $p=1-e^{-\beta}.$
\end{Proposition}
\begin{proof}
For simplicity, we write $\mu = \mu_{X}$ in the following calculation.
\begin{eqnarray*}
\tau_{\beta,q}\paren{\gamma}&=&\nu\paren{W_{\gamma}=1}-\frac{1}{q}\\
&=& \paren{\nu\paren{W_{\gamma}=1\mid V_{\gamma}}-\frac{1}{q}}\mu\paren{V_{\gamma}} + \paren{\nu\paren{W_{\gamma}=1\mid V_{\gamma}^c}-\frac{1}{q}}\mu\paren{V_{\gamma}^c}\\
&=& \paren{1-\frac{1}{q}}\mu\paren{V_{\gamma}}\,,
\end{eqnarray*}
where we used Proposition~\ref{prop:wilsondist} in the last step.
\end{proof}

\subsection{Proof of Theorem~\ref{thm:areaperimeter}}
By Theorem~\ref{thm:comparison}, a sharp phase transition for Wilson loop variables is equivalent to one for the topological variables $V_\gamma$ in the corresponding random-cluster model. Such a result is unknown for two-dimensional percolation in four dimension, let alone the random-cluster model, though we prove a ``qualitative'' analogue in Section~\ref{sec:homperc} below. We can, however, show a partial result by comparison to plaquette percolation.  First, we prove  a stochastic domination result for the plaquette random-cluster model which is a direct generalization of the corresponding classical result~\cite{fortuin1972random}.
 
\begin{Lemma}
\label{lemma:stochasticdom}
For any finite cubical complex $X,$ $\mu_{X}$ is stochastically decreasing in $q\geq 1$ for fixed $p.$ On the other hand, if we fix $\hat{p}=\frac{p/q}{1-p+p/q}$ then $\mu_{X,\hat{p},q,i}$ is stochastically increasing in $q\geq 1.$

\end{Lemma}
\begin{proof}
This is a consequence of the fact that adding an $i$-plaquette can only reduce $\mathbf{b}_{i-1}$ by one or leave $\mathbf{b}_{i-1}$ unchanged.

Recall that by Equation~\ref{eq:openprob}, $\hat{p}$ is the conditional probability that a plaquette is open given that it reduces $\mathbf{b}_{i-1}$ by one when leaving the states of the other plaquettes unchanged. First fix $p.$ Since $\hat{p}$ is decreasing in $q$ and the conditional probability that a plaquette is open given that it does not kill an $(i-1)$-cycle is constant in $q$ (equaling $p$), an application of Theorem~\ref{thm:stochdom} shows that $\mu_{X}$ is stochastically decreasing in $q.$ Now fix $\hat{p}.$ Then $p$ is decreasing as a function of $q,$ so again applying Theorem~\ref{thm:stochdom} to the conditional probabilities given by Equation~\ref{eq:openprob} yields that $\mu_{X}$ is stochastically increasing in $q.$
\end{proof}

A similar statement for $\mu_{\Z^d}$ follows immediately since it is a local limit of random-cluster measures on finite cubical complexes.

\begin{Corollary}\label{cor:zddom}
$\mu_{\Z^d}$ is stochastically decreasing in $q\geq 1$ for fixed $p.$ If we fix $\hat{p}=\frac{p/q}{1-p+p/q},$ $\mu_{\Z^d,\hat{p},q,i}$ is stochastically increasing in $q\geq 1.$
\end{Corollary}

Next, we require two more results of~\cite{ACCFR83} for plaquette percolation.

  \begin{Theorem}[Aizenman, Chayes, Chayes, Fr\"olich, Russo~\cite{ACCFR83}]
  \label{thm:accfrhigher}
  Consider $i$-dimensional Bernoulli plaquette percolation on $\Z^d.$  There exist finite, positive constants $\hat{c}_1=\hat{c}_1(p,d,i),\hat{c}_2=\hat{c}_2(p,d,i)>0$ so that, for hyperrectangular $(i-1)$-boundaries in $\Z^d,$
  \begin{equation}
 \label{eq:wilsoninequalities1_percolation}
 \exp(-\hat{c}_{1}\mathrm{Area}(\gamma) \leq \mathbb{P}_p(V_\gamma) \leq  \exp(-\hat{c}_2 \mathrm{Per}(\gamma))\,.
\end{equation}
  
Furthermore, there are constants $0<\tilde{p}_1\paren{d,i}\leq \tilde{p}_2\paren{d,i}<1$   so that  
  \begin{align*}
&-\frac{\log\paren{\mathbb{P}_p(V_\gamma)}}{\mathrm{Area}(\gamma)}\;\; \rightarrow \;\; \hat{c}_1 &&& \quad \text{if } p < \tilde{p}_1\paren{d,i} \\
&-\frac{\log\paren{\mathbb{P}_p(V_\gamma)}}{\mathrm{Per}(\gamma))}\;\;=\;\;\Theta\paren{1} &&&\quad \text{if }  p > \tilde{p}_2\paren{d,i}\,.
\end{align*}   
  \end{Theorem}
  
  The next lemma for the plaquette random-cluster model follows from the same tiling argument used to show Proposition 2.4 of~\cite{ACCFR83}, which only requires positive association.

\begin{Lemma}
\label{lemma:areaconvergence}
Let $\set{\gamma_n}$ be a sequence of hyperrectangular $(i-1)$-boundaries whose dimensions diverge with $n.$ Then
\[\lim_{n\rightarrow\infty} \frac{\mu_{\Z^d}\paren{V_{\gamma_n}}}{\mathrm{Area}(\gamma_n)}\] 
converges, where $\mu_{\Z^d}$ denotes an infinite volume random-cluster measure, constructed with either free or wired boundary conditions.
\end{Lemma}

\begin{proof}[Proof of Theorem~\ref{thm:areaperimeter}]
Let $p_1\paren{\beta}=1-e^{-\beta}$ and $p_2\paren{\beta}=\frac{p_1\paren{\beta}/q}{1-p_1\paren{\beta}+p_1\paren{\beta}/q}$. It follows from Corollaries~\ref{cor:VW} and~\ref{cor:zddom}  that 
\[\mu^{\mathbf{f}}_{\Z^d,p_1\paren{\beta},q,i}\paren{V_\gamma}\leq \mathbb{E}_{\nu^{\mathbf{f}}}\paren{W_{\gamma}}\leq \mu^{\mathbf{f}}_{\Z^d,p_2\paren{\beta},q,i}\paren{V_{\gamma}}\,.\]

Then the inequalities 
\begin{equation*}
 \exp(-\oldconstant{const:3} \mathrm{Area}(\gamma)) \leq \mathbb{E}_{\nu^{\mathbf{f}}}(W_\gamma) \leq  \exp(-\oldconstant{const:4} \mathrm{Per}(\gamma))\,.
\end{equation*}
follow from the corresponding statements in Equation~\ref{eq:wilsoninequalities1_percolation}. In addition, we may set $\beta_1=-\log\paren{1-\hat{p}}$ where
\[\hat{p}=\frac{\tilde{p}_1\paren{i,d} q}{\tilde{p}_1\paren{i,d} (q-1)+1}\]
and $\tilde{p}_1\paren{i,d}$ is given by Lemma~\ref{thm:accfrhigher}. Finally, fix $\beta_2=-\log\paren{1-\tilde{p}_2\paren{i,d}}$ and note that the existence of the constant  $-\oldconstant{const:3}$ follows from Lemma~\ref{lemma:areaconvergence}. Thus, we have the desired statement:
\begin{align*}
&-\frac{\log\paren{\mathbb{E}_{\nu^{\mathbf{f}}}(W_\gamma)}}{\mathrm{Area}(\gamma)}\;\; \rightarrow \;\; \oldconstant{const:3}  && \quad \text{if } \beta<\beta_1 \\
&-\frac{\log\paren{\mathbb{E}_{\nu^{\mathbf{f}}}(W_\gamma)}}{\mathrm{Per}(\gamma))}\;\;=\;\;\Theta\paren{1} &&\quad \text{if } \beta >\beta_2\,.
\end{align*}   
\end{proof}

\section{Phase Transitions for Homological Percolation}
\label{sec:homperc}

In this section, we prove the existence of sharp phase transitions for the $q$-state $i$-dimensional plaquette random-cluster model on $\T^d$ in the sense of homological percolation. First, we recall the definition of homological percolation. Let $P$ be an $i$-dimensional subcomplex of $\T^d_N$ and let $\phi_*:H_i\paren{P;\,\F}\rightarrow H_i\paren{\T^d;\,\F}$ be the map induced by inclusion. Then $b_i=\rank\phi_*$ counts the number of giant cycles of $P.$ A homological percolation phase transition occurs at $p_c$ if $b_i=0$ with high probability if $p<p_c,$ and $b_i$ attains its maximum possible value, $\binom{d}{i},$ with high probability if $p>p_c.$

\subsection{Interpretation for Potts Lattice Gauge Theory}
The homological percolation transition has two interpretations for Potts lattice gauge theory. The first is in terms of the behavior of the plaquette Swendsen--Wang algorithm, and was discussed in Section~\ref{sec:sw}. The other is related to the behavior of generalized  Polyakov loop variables, random variables of the form $W_\gamma\paren{f}$ where $\gamma$ is a giant $(i-1)$-cycle. 

\begin{Proposition}
Let $i \geq 1$ and let $X \subset \T^d_N$ be a subcomplex containing a giant $i$-cycle. Then there exist Polyakov loops $\gamma_1,\gamma_2 \in Z_{i-1}\paren{X;\,\F_q}$ so that $d_{\T^d_n}\paren{\mathrm{supp}\paren{\gamma_1},\mathrm{supp}\paren{\gamma_2}} \geq N/2-1$ and $\brac{\gamma_1} = \brac{\gamma_2} \in H_{i-1}\paren{X;\,\F_q}.$
\end{Proposition}

\begin{proof}
Let $\rho \in H_i\paren{X;\,\F_q}$ be a giant $i$-cycle. First, we show that a transverse $(d-1)$-dimensional slice of $\rho$ contains a giant $(i-1)$-cycle. Let $T_1 = \T^d_N \cap \set{x_1=1/2}$ and $T_2 = \T^d_N \cap \set{x_1=\floor{N/2} + 1/2}.$ Subdivide the cubical complex $\T_N^d$ to add an $(k-1)$-plaquette at each nonempty intersection of a $k$-plaquette of $\T^d_N$ with $T_1$ or $T_2$ for each $k \leq d.$ Call this subdivision $T,$ and give $T_1,$ $T_2,$ and $X$ the resulting cubical complex structures.

Without loss of generality assume that $\rho$ is homologous to a sum of the standard generators of the torus, including at least one that is orthogonal to $T,$ i.e. a standard generator that is a product of $i$ many $S^1$ factors, one of which is in the first coordinate direction (see the discussion of the homology of the torus at the end of Section~\ref{sec:hom}). If $\rho = \sum_{\sigma \in X^i} c_{\sigma} \sigma$ (remember that $X^i$ now has subdivided cells that were not originally in $\T^d_N)$), let 
\[\gamma_1= \sum_{\sigma \in X^i} c_{\sigma} \paren{\sigma \cap T_1}\,.\]
Note that each $i$-plaquette of $X^i$ intersects $T_1$ in either an $(i-1)$-plaquette or not at all. Also, let $\gamma_2$ be constructed as $\gamma_1$ with $T_1$ replaced by $T_2.$ Observe that the supports of $\gamma_1$ and $\gamma_2$ are at distance at least $N/2-1$ apart.

Consider $D \coloneqq \T^d_N \setminus \paren{T\cup T'}.$ Since $D$ is disconnected, we can write it as a union of its connected components $D = D_1 \cup D_2.$ Let

Let \[\rho_1 \coloneqq \sum_{\sigma \cap D_1 \neq \emptyset} c_\sigma \sigma\] and \[\rho_2 \coloneqq \sum_{\sigma \cap D_2 \neq \emptyset} c_\sigma \sigma\,.\] Then $\rho_1$ is an $i$-chain with boundary $\gamma_1 + \gamma_2,$ so $\gamma_1$ and $\gamma_2$ are homologous. Then it only remains to show that they are nontrivial giant cycles, as we can then shift them by $\paren{-1/2,0,\ldots,0}$ to obtain the desired Polyakov loops.

We claim that $\gamma_1$ and $\gamma_2$ are giant $(i-1)$-cycles within $T_1$ and $T_2$ respectively. Since $\gamma_1$ and $\gamma_2$ are homologous, it is enough to show that at least one of them is a giant cycle. Suppose that neither are, so there are $i$-chains $\alpha_1,\alpha_2 \in H_i\paren{\T^d_N;\,\F_q}$ so that $\gamma_1 = \partial \alpha_1$ and $\gamma_2 = \partial \alpha_2.$ so $\beta_1 \coloneqq \paren{\rho_1} - \alpha_1 - \alpha_2$ is an $i$-cycle. Moreover, since $\beta_1$ is contained in a subset of $\T^d_N$ that deformations retracts to $T_1,$ namely $T_1 \cup D_1 \cup T_2,$ it is homologous within $\T^d_N$ to an $i$-cycle $\hat{\beta}_1$ contained in $T_1.$ Similarly, we set $\beta_2 \coloneqq \paren{\rho \cap E_2} - \alpha_1 - \alpha_2,$ and can find $\hat{\beta}_2$ contained in $T$ that is homologous to $\beta_2$ within $\T^d_N.$ Then we have 
\[\brac{\rho} = \brac{\hat{\beta}_1 + \hat{\beta}_2} \in H_{i}\paren{\T^d_N;\,\F_q}\,.\]
But the $(d-1)$-tori $T_1$ and $T_2$ can only contain certain giant cycles described by the K\"{u}nneth formula. In particular, they cannot contain orthogonal homology classes, contradicting the assumption about $\rho$ made earlier.
\end{proof}

The following is now a corollary of Theorems~\ref{thm:half} and ~\ref{thm:weak}.
\begin{Corollary}
Let $1\leq i<d,$ let $q$ be an odd prime, and set
\[\beta_{\mathrm{surf}} = \beta_{\mathrm{surf}}\paren{q,d,N} \coloneqq -\log\paren{1-\lambda\paren{q,d,N}}\,.\]
If $\beta<\beta_{\mathrm{surf}},$ then $(i-1)$-dimensional then, with high probability, $q$-state Potts lattice gauge theory has Polyakov loops which all take the same value and rule a giant cycle of the corresponding plaquette random cluster model. When $d=2i,$ we can take $\beta_{\mathrm{surf}}=\beta_{\mathrm{sd}}=\log\paren{1+\sqrt{q}}.$
\end{Corollary}

\subsection{Some Probabilistic Tools}

Though the random-cluster model is more difficult to work with than its Bernoulli counterpart, useful tools have been developed that work in both settings, particularly in recent years. We will use several of these results, both for graph and for higher dimensional cubical complexes.

When comparing different critical probabilities, it is essential to be able to compare the probabilities of events in different percolation models. This is easy to do when we can show that one model has strictly more open edges or plaquettes in a probabilistic sense. Let $E$ be a finite set and let $\mu_1,\mu_2$ be probability measures on $\Omega = \set{0,1}^E.$ We say that $\mu_1$ is stochastically dominated by $\mu_2$ if there is a probability measure $\kappa$ on $\Omega \times \Omega$ with first and second marginals $\mu_1$ and $\mu_2$ so that 
\[\kappa\paren{\set{\paren{\omega_1,\omega_2}: \omega_1 \leq \omega_2}} = 1\,.\]
In this case we write $\mu_1 \leqst \mu_2.$ We will use a version of Holley's theorem that gives a criterion for stochastic domination in terms of conditional probabilities. This can be found as Theorem 2.3 of~\cite{grimmett2006random}.

\begin{Theorem}[Holley]\label{thm:stochdom}
Let $E$ be a finite set and let $\mu_1,\mu_2$ be strictly positive probability measures on $\Omega = \set{0,1}^E.$ Suppose that for each pair $\xi,\zeta \in \Omega$ with $\xi \leq \zeta$ and each $e_0 \in E,$
\begin{align*}
    &\mu_1\paren{\omega\paren{e_0} = 1 : \omega\paren{e} = \xi\paren{e} \text{ for all } e \in E \setminus \set{e_0}}\\ 
    &\qquad\leq \mu_2\paren{\omega\paren{e_0} = 1 :  \omega\paren{e} = \zeta\paren{e} \text{ for all } e \in E \setminus \set{e_0}}\,.
\end{align*}
Then $\mu_1 \leqst \mu_2.$
\end{Theorem}

We will also employ a sharp threshold theorem for monotonic measures due to Graham and Grimmett~\cite{graham2006influence}.

\begin{Theorem}[Graham and Grimmett]\label{thm:grahamgrimmett}
There exists a constant $0 < \newconstant\label{const:6} < \infty$ so that the following holds. Let $N \geq 1,$ $I = \set{1,\ldots,N},$ $\Omega = \set{0,1}^N,$ and let $\mathcal{F}$ be the set of subsets of $\Omega.$ Let $A \in \mathcal{F}$ be an increasing event. Let $\mu$ be a positive monotonic probability measure on $(\Omega, \mathscr{F}).$ Let $X_i = \omega\paren{i}$ and set $p = \mu\paren{X_i=1}.$ If there exists a subgroup $\mathcal{A}$ of the symmetric group on $N$ elements $\Pi_N$ acting transitively on $I$ so that $\mu$ and $A$ are $\mathcal{A}$-invariant, then 
\begin{align*}
    \frac{d}{dp}\mu_p\paren{A} \geq \frac{\oldconstant{const:6}\mu_p\paren{X_1}\paren{1-\mu_p\paren{X_1}}}{p\paren{1-p}}\min\set{\mu_p\paren{A},1-\mu_p\paren{A}}\log N\,.
\end{align*}
\end{Theorem}

In the 1-dimensional case, we will use the recent breakthrough results of Duminil-Copin, Raoufi, and Tassion characterizing the subcritical and supercritical regimes of the random-cluster model.

\begin{Theorem}[Duminil-Copin, Raoufi, Tassion]\label{thm:expdecay}
Fix $d \geq 2$ and $q \geq 1.$ Let $\theta(p) = \mu^{\mathbf{w}}_{\Z^d,p,q,1}\paren{0 \leftrightarrow \infty}.$ Then 
\begin{itemize}
    \item there exists a $\newconstant\label{const:7} > 0$ so that $\theta\paren{p} \geq \oldconstant{const:7}\paren{p-\hat{p}_c}$ for any $p \geq p_c$ sufficiently close to $p_c;$
    \item for any $p < \hat{p}_c,$ there exists a $c_p$ so that for every $n \geq 0,$
    \[ \mu^{\mathbf{w}}_{\Lambda_n,p,q,1}\paren{0 \leftrightarrow \partial \Lambda_n} \leq \exp\paren{-c_p n}\,.\]
\end{itemize}
\end{Theorem}

Lastly, we will use a result of Kahle and the authors on vector-valued random variables whose range is an irreducible representation of the symmetry group of the probability space.

\begin{Lemma}[Duncan, Kahle, and Schweinhart]\label{lemma:spinning}
Let $V$ be a finite dimensional vector space and $Y$ be a set. Let $\mathcal{A}$ be the lattice of subspaces of $V.$ Suppose $h : \mathcal{P}\paren{Y} \to \mathcal{A}$ is an increasing function, i.e. if $A \subset B$ then $h\paren{A} \subset h\paren{B}.$ Let $\mathcal{G}$ be a finite group which acts on both $Y$ and $V$ whose action is compatible with $h.$ That is, for each $g \in \mathcal{G}$ and $v\in V,$  $g\paren{h\paren{v}} = h\paren{gv}.$ Let $X$ be a $\mathcal{P}\paren{Y}$-valued random variable with a $\mathcal{G}$-invariant distribution that is positively associated, meaning that increasing events with respect to $X$ are non-negatively correlated. Then if $V$ is an irreducible representation of $\mathcal{G}$, there are positive constants $\newconstant\label{const:8}, \newconstant\label{const:9}$ so that  $$\mathbb{P}_p\paren{h(X) = V} \geq \oldconstant{const:8}\mathbb{P}_p\paren{h(X) \neq 0}^{\oldconstant{const:9}}\,,$$ where $\oldconstant{const:8}$ only depends on $\mathcal{G}$ and $\oldconstant{const:9}$ only depends on $\dim V.$
\end{Lemma}

\subsection{Proof of Theorem~\ref{thm:half}}

We recall that $A$ is the event that the map on homology induced by the inclusion $P \hookrightarrow \T^d$ is nonzero and $S$ is the event that the same map is surjective. Note that by the intermediate value theorem we may define a threshold function $\lambda=\lambda\paren{N}$ so that $\mu_{\T^d_N,\lambda,q,i}\paren{A} = 1/2$ for each $N.$ We begin by showing a general threshold result for homological percolation. The proof is very similar to that in~\cite{duncan2020homological}.

\begin{Proposition}\label{prop:sharp}
Let $q$ be an odd prime and $1\leq i \neq d-1.$ For any $\epsilon > 0,$ we have
\begin{align*}
    \mu_{\T^d_N,\lambda-\epsilon,q,i}\paren{A} \to 0
\end{align*}
and
\begin{align*}
    \mu_{\T^d_N,\lambda+\epsilon,q,i}\paren{S} \to 1
\end{align*}
as $N \to \infty.$
\end{Proposition}

\begin{proof}
We prove the second statement first. Note that increasing events with respect to $\mu_{\T^d_N}$ are positively correlated by Theorem~\ref{thm:i-FKG} and that the group of symmetries of the torus contains an irreducible representation of $H_i\paren{\T^d_N}$ (see Proposition 13 of~\cite{duncan2020homological} for a proof). Then by Lemma~\ref{lemma:spinning},
\begin{align*}
    \mu_{\T^d_N}\paren{S} \geq \oldconstant{const:8}\mu_{\T^d_N}\paren{A}^{\oldconstant{const:9}}\,,
\end{align*}
where $\oldconstant{const:8},\oldconstant{const:9}$ do not depend on $N.$
In particular there is a $\delta>0$ so that $\mu_{\T^d_N,\lambda,q,i}\paren{S} \geq \delta$ for all $N.$

Now we show that $\mu_{\T^d_N}\paren{S}$ has a sharp threshold. By Theorem~\ref{thm:i-FKG}, $\mu_{\T^d_N}$ satisfies the FKG lattice condition and is thus monotonic~\cite{grimmett2004random}. Then since the symmetries of $\T^d_N$ act transitively on the plaquettes, we apply Theorem~\ref{thm:grahamgrimmett} to obtain
\begin{align*}
    \frac{d}{dp}\mu_{\T^d_N}\paren{S} &\geq \frac{\oldconstant{const:6}}{q}\min\set{\mu_{\T^d_N}\paren{S},1-\mu_{\T^d_N}\paren{S}}\log \abs{F^i_N}\,.
\end{align*}
In particular, for $p$ close to $\lambda,$ $\frac{d}{dp}\mu_{\T^d_N,p,q,i}\paren{S} \geq \frac{\oldconstant{const:6}\delta}{q}\log \abs{F^i_N}.$ By integrating this inequality, we have $\mu_{\T^d_N,\lambda + \epsilon,q,i}\paren{S} \to 1$ as $N \to \infty.$ 

To obtain the first statement, we apply the same argument to the dual system. Since $\mu_{\T^d_N,p,q,i}\paren{A} = 1/2,$ it follows that $\mu_{\T^d_N,\paren{\lambda}^*,q,i}\paren{A} \geq 1/2$ by Equation~\ref{eq:2}. Then since $p^*$ is continuous as a function of $p,$ the above argument shows that $\mu_{\T^d_N,\paren{\lambda-\epsilon}^*,q,i}\paren{S} \to 1$ as $N\to \infty.$ By another application of Equation~\ref{eq:2}, we then have $\mu_{\T^d_N,\lambda-\epsilon,q,i}\paren{A} \to 0$ as $N \to \infty.$

\end{proof}

\begin{Corollary}\label{cor:sharp}
If $p_0 > p_u\paren{q,i,d},$ then 
\begin{align*}
    \mu_{\T^d_N,p_0,q,i}\paren{S} \to 1
\end{align*}
as $N \to \infty.$ If $p_0 < p_l\paren{q,i,d},$ then
\begin{align*}
    \mu_{\T^d_N,p_0,q,i}\paren{A} \to 0
\end{align*}
as $N \to \infty.$
\end{Corollary}

It is now straightforward to prove that the $i$-dimensional random-cluster model in $\T^{2d}$ exhibits a sharp phase transition at the self-dual point.

\begin{proof}
By self duality and Equation~\ref{eq:2}, 
\[\mu_{\T^d_N,p_{\mathrm{sd}},q,i}\paren{A} \geq 1/2\,.\]
In particular, $p_u \leq p_{\mathrm{sd}}.$ Then by monotonicity and Corollary~\ref{cor:sharp}, \[\mu_{\T^d_N}\paren{S} \to 1\]
as $N \to \infty$ for all $p > p_{\mathrm{sd}}.$ Since $p^*$ is decreasing as a function of $p$ with fixed point $p_{\mathrm{sd}},$ applying Equation~\ref{eq:2} again gives 
\[\mu_{\T^d_N}\paren{A} \to 1\]
as $N \to \infty$ for all $p < p_{\mathrm{sd}}.$
\end{proof}

\subsection{Proof of Theorem~\ref{thm:one}} Next, we will prove our theorem on the homological percolation transition for the $1$-dimensional and $(d-1)$-dimensional random-cluster models in $\T^d.$ Throughout this subsection, we fix $i=1,$ and we borrow terminology from percolation on graphs. Let $G$ be a graph and let $V\paren{G}$ and $E\paren{G}$ be its vertex and edge sets respectively. Designate the edges of $E\paren{G}$ as open or closed according to some probability measure. For $v,w \in V\paren{G},$ we write $v \leftrightarrow w$ if $v$ is connected to $w$ by a path of open edges. For a vertex subset $V' \subset V\paren{G},$ we write $v \leftrightarrow_{V'} w$ if $v$ is connected to $w$ by a path of open edges that only passes through vertices of $V'.$



We will prove the statements of Theorem~\ref{thm:one} for the case $i=1$ in each regime separately. The case $i=d-1$ then follows immediately from Equation~\ref{eq:2}. The subcritical case is a straightforward application of the exponential decay of the size of components.

\begin{proof}[Proof of Theorem~\ref{thm:one} (subcritical regime)]
Let $p<\hat{p}_c$ and let $M=\floor{N/2}.$ For a vertex $x$ of $\mathbb{T}_N^d,$ let $A_x$ the event that a giant $1$-cycle passes through $x.$ Since $[-M,M]^d=\Lambda_M$ is contractible in $\T^d_N,$ $A_0 \subset \set{0 \leftrightarrow_{\T^d_N} \partial \Lambda_M}.$ Moreover, $\mu^{\mathbf{w}}_{\Lambda_M}$ stochastically dominates $\mu_{\T^d_N} |_{\Lambda_M},$ so by Theorem~\ref{thm:expdecay} and translation invariance,
 \begin{equation}
     \label{eq:pboundsn1}
     \mu_{\T^d_N}\paren{A_x}\leq \mu^{\mathbf{w}}_{\Lambda_M}\paren{0 \leftrightarrow \Lambda_M} \leq \exp\paren{c_p M}
 \end{equation}
for all vertices $x$ of $\mathbb{T}_N^d.$ 

Let $X$ be the number of vertices in $\mathbb{T}_N^d$ that are contained in a giant $1$-cycle.
Since $A=\set{X\geq 1},$

\begin{align*}
    \mu_{\T^d_N}\paren{A}&=\;\;\mu_{\T^d_N}\paren{X\geq 1}\\
    &\leq\;\; \mathbb{E}_{\mu}\paren{X} && \text{by Markov's Inequality}\\
    &=\;\; \sum_{x\in \mathbb{T}_N^d}\mu_{\T^d_N}\paren{A_x} \\
    & \leq\;\;  N^d e^{-c_p M} &&\text{by Equation ~\ref{eq:pboundsn1}}\\
    & = \;\; N^d e^{-c_p \floor{N/2}} \to 0
\end{align*}
as $N\rightarrow\infty$.

\end{proof}

The supercritical case is more involved and will require assuming Conjecture~\ref{con:slab} --- the continuity of the critical probabilities for the existence of an infinite component for the $1$-dimensional random-cluster model in slabs --- in order to obtain a sharp threshold. Our strategy will be to \say{pin} together paths from supercritical percolation in large slabs in order to form a giant $1$-cycle. 

Let 
\[\Lambda_{n,k} \coloneqq \brac{-n,n}^2 \times \brac{-k,k}^{d-2} \cap \Z^d \subset S_k\,.\]
Let $D_{n,k} \coloneqq \set{v \in \Lambda_{n,k} : v \sim u \text{ for some } u \in S_k \setminus \Lambda_{n,k}}$ be the boundary of $\Lambda_{n,k}$ in $S_k.$ 

\begin{Lemma}\label{lemma:slabsharp}
Fix $q \geq 1, d \geq 2,$ and  $k \geq 1.$ There is a $\newconstant\label{const:10}>0$ so that for any $p > p_c\paren{S_k}$ sufficiently close to $p_c\paren{S_k}$ and any $n$ sufficiently large, 
\[\mu_{\Lambda_{n,k}}^{\mathbf{f}}\paren{0 \leftrightarrow D_{n,k}} \geq \oldconstant{const:10}\paren{p-p_c\paren{S_k}}\,.\]
\end{Lemma}

The proof follows from the proof of Theorem~\ref{thm:expdecay} (Theorem 1.2 in~\cite{duminil2019sharp}), replacing $\Lambda_n$ with $\Lambda_{n,k}.$ With Lemma~\ref{lemma:slabsharp} in hand, we can give a lower bound for the probability that two vertices are connected within a box contained in the slab $S_k.$

\begin{Lemma}\label{lemma:boxoverlap}
Fix $q \geq 1, d \geq 2, k \geq 1,$ and $p > p_c\paren{S_k}.$  Let $\Lambda = [-3n,3n]\times [-2n,4n] \times [-k,k]^{d-2} \cap \Z^d.$ There is a constant $\newconstant\label{const:11}>0$ not depending on $n$ so that 
\[\mu_{\Lambda}^{\mathbf{f}}\paren{\paren{-n,0,\ldots,0} \leftrightarrow \paren{n,0,\ldots,0}} \geq \oldconstant{const:11}\,.\]
\end{Lemma}

\begin{proof}
Consider the random-cluster model $\mu_{\Lambda}^{\mathbf{f}}.$ Define three crossing events 
\[V_+ \coloneqq \set{\paren{\vec{0}} \leftrightarrow_{\Lambda_{2n}} \set{\paren{j,2n,x_3,\ldots,x_d} : 0 \leq j \leq 2n, -k \leq x_3,\ldots,x_d \leq k}}\,,\]
\[V_- \coloneqq \set{\paren{\vec{0}} \leftrightarrow_{\Lambda_{2n}} \set{\paren{-j,2n,x_3,\ldots,x_d} : 0 \leq j \leq 2n, -k \leq x_3,\ldots,x_d \leq k}}\,,\]
and
\[U_+ \coloneqq \set{\paren{\vec{0}} \leftrightarrow_{\Lambda_{2n}} \set{\paren{j,-2n,x_3,\ldots,x_d} : 0 \leq j \leq 2n, -k \leq x_3,\ldots,x_d \leq k}}\,.\]
For an event $A(\vec{0})$ defined in relation to the origin $\vec{0},$ let $A(\vec{x})$ denote the event $A+\vec{x}$ obtained by translating the origin to $\vec{x}.$ By Lemma~\ref{lemma:slabsharp} and symmetry, there is a $\newconstant\label{const:12}>0$ not depending on $n$ so that 
\begin{equation}
\label{eq:boxoverlap1}
\mu_{\Lambda}^{\mathbf{f}}\paren{V_+} = \mu_{\Lambda}^{\mathbf{f}}\paren{V_-} = \mu_{\Lambda}^{\mathbf{f}}\paren{U_+}=\oldconstant{const:12}\,.
\end{equation}

Let $\vec{v} = \paren{-n,0,\ldots,0}$ and $\vec{w} = \paren{n,0,\ldots,0}.$ Denote the projection onto the first two coordinates by $\pi_{12}: \Z^d \to \Z^2.$ Our aim will be to create overlapping paths (i.e. paths with intersecting images under $\pi_{12}$) containing $v$ and $w,$ which are then close enough to be connected with positive probability. 

First, we will construct two paths that project to a horizontal crossing of of  $\brac{-2n,4n}^2$ under $\pi_{12}.$  Fix an arbitrary ordering of the finite paths in $\Z^d.$ Assuming that  $V_+\paren{\vec{v}}$ occurs, let $\hat{\vec{v}}$ be the endpoint of the minimal path witnessing that event and let  $\vec{v}'=\pi_{12}\paren{\hat{\vec{v}}}\times \paren {0,\ldots,0}.$ Then, the event 
\begin{equation}
\label{eq:boxoverlap2}
H= U_+\paren{\vec{v}}\cap V_+\paren{\vec{v}}\cap V_{-}\paren{\vec{v}'}\,,
\end{equation}
ensures the existence of a path through $\vec{v}$ with a single gap at $\vec{v}'$ which projects to a vertical crossing, as illustrated in in Figure~\ref{fig:boxoverlap}.

Next, let $I$ be the event $H$ rotated $\pi/2$ radians counterclockwise in the first two coordinates about the point $\paren{0,n,0,\ldots,0}.$ By symmetry, $\mu_{\Lambda}^{\mathbf{f}}\paren{I}=\mu_{\Lambda}^{\mathbf{f}}\paren{H},$ and that event ensures the existence of a path through $\vec{w}$ with a single gap which projects to a horizontal crossing of $\brac{-2n,4n}^2.$

\begin{figure}
    \centering
    \includegraphics[width=.5\textwidth]{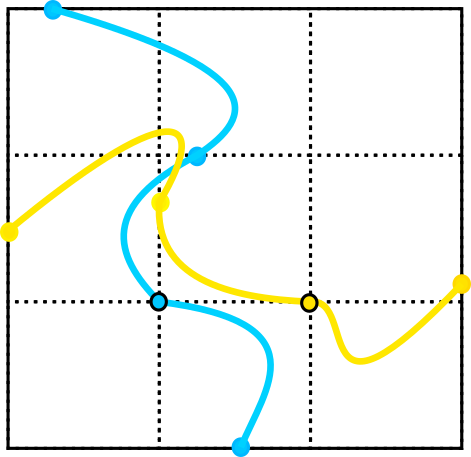}
    \caption{The two paths constructed in Lemma~\ref{lemma:boxoverlap}. The blue path is a projection of a witness for the event $H$ and the yellow path is a projection of a witness for $I.$ The circled blue dot is $\vec{v}$ and the circled orange dot is $\vec{w}.$ Both paths are shown in $\Lambda$ with a grid spacing of $2n$ for clarity.}
    \label{fig:boxoverlap}
\end{figure}

Now, assume that $H\cap I$ occurs so $\pi_{12}\paren{\vec{v}}$ and $\pi_{12}\paren{\vec{w}}$ are connected in $\pi_{12}\paren{P}.$ We count how many edges we would need to add in order to connect the overlapping segments in $P$, thus connecting $\vec{v}$ to $\vec{w}.$ Any two points in $S_k$ with the same image under $\pi_{12}$ are at graph distance at most $2k\paren{d-2}.$ There is one point of overlap between $\hat{\vec{v}}$ and $\vec{v}',$ and one more from the symmetric event for $\vec{w}.$ In addition, we have a third point of overlap between the horizontal and vertical crossings. As such, we can connect $\vec{v}$ to $\vec{w}$ using at most $6k\paren{d-2}$ additional open edges.

Let $p > r > p_c\paren{S_k}.$ Given a configuration $\omega,$ define 
\[S^m(\omega) = \set{\omega': \sum_{e \in E\paren{G}} \abs{\omega\paren{e} - \omega'\paren{e}} \leq m}\,.\]
For an event $E,$ we then define
\[E^m(E) = \set{\omega: S^m(\omega) \cap E \neq \emptyset}\,.\]Then by Theorem 3.45 of~\cite{grimmett2004random} and the FKG inequality there is a constant $\newconstant\label{const:13}$ so that
\begin{align*}
    &\mu_{\Lambda}^{\mathbf{f}}\paren{\vec{v} \leftrightarrow \vec{w}}\\
    &\qquad\geq \oldconstant{const:13}^{6k\paren{d-2}}\mu_{\Lambda}^{\mathbf{f}}\paren{E^{\paren{6k\paren{d-2}}}\paren{\vec{v} \leftrightarrow \vec{w}}}\\
    &\qquad\geq \oldconstant{const:13}^{6k\paren{d-2}}\mu_{\Lambda}^{\mathbf{f}}\paren{H \cap I}\\
    &\qquad\geq \oldconstant{const:12}^{6k\paren{d-2}} \oldconstant{const:13}^{6k\paren{d-2}}\,,
\end{align*}
where we used Equations~\ref{eq:boxoverlap1} and~\ref{eq:boxoverlap2} in the final step. Since this bound does not depend on $n$ we are done.

\end{proof}


\begin{proof}[Proof of Theorem~\ref{thm:one} (supercritical regime)]
Let $p > p' > p_c^{\mathrm{slab}}.$ Then there is a $k$ such that $p' > p_c\paren{S_k}.$ Let $N \geq 2k.$ We will construct a giant cycle in percolation with parameter $p'$ by using Lemma~\ref{lemma:boxoverlap} to connect the centers of four pairwise overlapping boxes in the torus, each of diameter $\floor{N/2}.$ If $N$ is not divisible by $2,$ the starting and ending points of this constructed path may not exactly match. However, they will be at graph distance at most $1,$ and are therefore connected with probability at least $\frac{p'}{q}.$ We may therefore assume that $N$ is divisible by $2$ in the remainder of the proof for simplicity. We apply Lemma~\ref{lemma:boxoverlap} to copies of 
\[\Lambda \coloneqq [-3N/8,3N/8]\times [-N/4,2N/4] \times [-k,k]^{d-2} \cap \Z^d\]
that are centered at $u_0 = \paren{N/8,N/8,\ldots,0},$ $u_1 = \paren{3N/8,N/8,\ldots,0},$ $u_2 = \paren{5N/8,N/8,\ldots,0},$ and $u_3 = \paren{7N/8,N/8,\ldots,0},$ to connect $v_0 = \vec{0},$ $v_1 = \paren{N/4,0,\ldots,0},$ $v_2 = \paren{N/2,0,\ldots,0},$ and $v_3 = \paren{3N/4,0,\ldots,0}.$ For convenience, let  $B(u,v,w)$ denote the event $u\leftrightarrow_{\lambda(w)} .$

If the events $B(v_0,v_1,u_0),$ $B(v_1,v_2,u_1),$ $B(v_2,v_3,u_2),$ and $B(v_3,v_0,u_3)$  all occur, then there is an open path that is homotopic to the standard generator of $H_1\paren{\T^d_N}$ contained in $\set{x_2=x_3=\ldots=x_d = 0}.$ Thus, 

\[H \supseteq B(v_0,v_1,u_0) \cap B(v_1,v_2,u_1)\cap B(v_2,v_3,u_2) \cap B(v_3,v_0,u_3)\,.\]

Let $\mu_{u_0}$ denote the measure $\mu_{\Lambda\paren{u_0},p'}^{\mathbf{f}}.$ We then apply the FKG inequality and Lemma~\ref{lemma:subcomplex} to bound
\begin{align*}
    \mu_{\T^d_N,p'}\paren{H} &\geq \mu_{\T^d_N,p'}\paren{B(v_0,v_1,u_0) \cap B(v_1,v_2,u_1)\cap B(v_2,v_3,u_2) \cap B(v_3,v_0,u_3)}\\
    &\geq     
    \mu_{u_0}\paren{v_0 \leftrightarrow v_1}\mu_{u_1}\paren{v_1 \leftrightarrow v_2}\mu_{u_2}\paren{v_2 \leftrightarrow v_3}\mu_{u_3}\paren{v_3 \leftrightarrow v_4}\\
    &\geq \oldconstant{const:11}^4\,.
\end{align*}
Note that the final bound is uniform in $N.$ Again using the FKG inequality, we obtain
\[\mu_{\T^d_N,p'}\paren{S} \geq \mu_{\T^d_N,p'}\paren{H}^{d} \geq \geq \oldconstant{const:11}^{4d}\,.\]
Then by applying Theorem~\ref{thm:grahamgrimmett} as in Proposition~\ref{prop:sharp}, we have
\[\mu_{\T^d_N,p}\paren{S} \to 1\]
as $N \to \infty.$

\end{proof}

\subsection{Proof of Theorem~\ref{thm:weak}} We conclude the paper by finishing our proof of our theorem which establishes the existence of threshold functions for homological percolation in general, and enumerates some of their properties. All that remains is to prove the duality and monotonicity statements. The former is topological, and does not require modification from the proof of the analogous statement in~\cite{duncan2020homological}.

\begin{Proposition}\label{prop:duality}
For any $d \geq 2$ and $1 \leq i \leq d-1,$
\[p_u\paren{q,i,d} = \paren{p_l\paren{q,d-i,d}}^*\,.\]
\end{Proposition}

\begin{proof}
This follows from Theorem ~\ref{theorem:duality}, Equation~\ref{eq:2}, and Corollary~\ref{cor:sharp} in a similar manner to Proposition 18 of~\cite{duncan2020homological}.
\end{proof}

We now turn to showing monotonicity in the critical probabilities. First we require a lemma comparing the $i$-dimensional plaquette random-cluster percolation on a cubical complex to the random-cluster model on a subcomplex.

\begin{Lemma}\label{lemma:subcomplex}
Fix $q \geq 1.$ Let $X,Y$ be finite $i$-dimensional cubical complexes with $X \subset Y$ and recall that $\mu_X$ and $\mu_Y$ are the $i$-dimensional plaquette random-cluster measures with parameters $p,q$ on $X$ and $Y$ respectively. Then $\mu_Y|_X$  stochastically dominates $\mu_X.$
\end{Lemma}

\begin{proof}
Let $\sigma \subset X$ be an $i$-cell. Let $S$ be a subcomplex of $Y,$ let $Y_\omega$ and let $R = S \cap X.$ Write $S^{\sigma} = S \cup \sigma$ and $S_{\sigma} = S \setminus \sigma$ and similarly $R^{\sigma} = R \cup \sigma$ and $R_{\sigma} = R \setminus \sigma.$ Since $Y$ has no $(i+1)$-cells,
\[\mathbf{b}_{i}\paren{R^{\sigma}} - \mathbf{b}_{i}\paren{R_{\sigma}} \leq \mathbf{b}_{i}\paren{S^{\sigma}} - \mathbf{b}_{i}\paren{S_{\sigma}}\,.\]
Then by the Euler--Poincar\'{e} formula we see that
\begin{align}\label{eq:subcomplexbetti}
\mathbf{b}_{i-1}\paren{R_{\sigma}} - \mathbf{b}_{i-1}\paren{R^{\sigma}} \geq \mathbf{b}_{i-1}\paren{S_{\sigma}} - \mathbf{b}_{i-1}\paren{S^{\sigma}}\,.
\end{align}
Recall that for any $i$-cell $\sigma \subset Z,$ the random subcomplex $T$ with distribution $\mu_Z$ satisfies
\begin{align}\label{eq:openprob}
    \mu_{Z}\paren{\sigma \in T \mid \mathbf{b}_{i-1}\paren{T_{\sigma}} - \mathbf{b}_{i-1}\paren{T^{\sigma}}} = \begin{cases}
  p & \mathbf{b}_{i-1}\paren{T_{\sigma}} - \mathbf{b}_{i-1}\paren{T^{\sigma}} = 0\\
  \frac{p/q}{1-p+p/q} &\mathbf{b}_{i-1}\paren{T_{\sigma}} - \mathbf{b}_{i-1}\paren{T^{\sigma}} = 1
\end{cases}
\end{align}

Then by positive association and Equations~\ref{eq:subcomplexbetti} and ~\ref{eq:openprob}, we may apply Theorem~\ref{thm:stochdom} to conlcude that $\mu_X$ stochastically dominates $\mu_Y|_X.$
\end{proof}

Finally, we compare percolation within subcomplexes of the torus in order to obtain the desired inequalities between critical probabilities.
\begin{Proposition}
For all $1 \leq i \leq d-1,$
\[p_u\paren{q,i,d} < p_u\paren{q,i,d-1}< p_u\paren{q,i+1,d}\,.\]
\end{Proposition}

\begin{proof}
The topological properties of any given configuration of plaquettes are identical to those discussed Proposition 24 of~\cite{duncan2020homological}. We will therefore only modify the probabilistic arguments as necessary.

Our first goal is to show
\[p_u\paren{q,i,d} < p_u\paren{q,i,d-1}.\]
Our strategy will be to define a sequence of models between the random-cluster model on $\T^d_N$ and $\T^{d-1}_N$ in which the giant cycle space of each model stochastically dominates the giant cycle space of the one before. More precisely, for a configuration of plaquettes $\omega,$ let $G\paren{\omega}$ be the associated subspace of giant cycles of $P\paren{\omega}$ in $H_i\paren{\T^d_N}.$ Then we say $\mu_1 \leq_G \mu_2$ if there is a coupling $\kappa$ of $\mu_1$ and $\mu_2$ so that \[\kappa\paren{\set{\paren{\omega_1,\omega_2} : G\paren{\omega_1} \subseteq G\paren{\omega_2}}} = 1\,.\]
Note that $\mu_1 \leqst \mu_2$ implies $\mu_1 \leq_G \mu_2.$

Let $\mathcal{T}_0 = \T^d_N,$ $\mathcal{T}_1 = \T^d_N \cap \set{x_1 \in [0,1]},$ and $\mathcal{T}_2 = \T^d_N \cap \set{x_1 = 0}.$ For $j=0,1,2,$ let $\mu_{\mathcal{T}_j}$ be so that $\mu_{\mathcal{T}_j}|_{\mathcal{T}_j}$ is the random-cluster model on $\mathcal{T}_j$ and $\mu_{\mathcal{T}_j}|_{\T^d_N \setminus \mathcal{T}_j}$ sets all plaquettes to be closed almost surely. By Lemma~\ref{lemma:subcomplex}, we have 
\[\mu_{\mathcal{T}_0} \geq_{\mathrm{st}} \mu_{\mathcal{T}_1}\,.\]
We now put a different measure $\mu_{\mathcal{T}_1}'$ on configurations that are closed outside $\mathcal{T}_1.$ Let $F_2$ be the set of $i$-cells of $\mathcal{T}_1$ contained in $\mathcal{T}_2$ and let $F_1$ be the rest of the $i$-cells of $\mathcal{T}_1.$ Let $\eta_1$ and $\eta_2$ be the number of open cells of $F_1$ and $F_2$ respectively. We set $\mu_{\mathcal{T}_1}'|_{\mathcal{T}_2} = \mu_{\mathcal{T}_2}|_{\mathcal{T}_2}.$ We then let $\mu_{\mathcal{T}_1}'|_{\mathcal{T}_1 \setminus \mathcal{T}_2}$ be independent Bernoulli plaquette percolation with probability $p/q$ and declare all other plaquettes closed. More explicitly,
\begin{align*}
    \mu_{\mathcal{T}_1}'\paren{\omega} \coloneqq \frac{1}{Z}p^{\eta_2\paren{\omega}}\paren{1-p}^{\abs{F_2} - \eta_2\paren{\omega}}q^{\mathbf{b}_{i-1}\paren{P_{\omega|_{\mathcal{T}_2}}}}\paren{\frac{p}{q}}^{\eta_1}\paren{1-\frac{p}{q}}^{\abs{F_1} - \eta_1}\,.
\end{align*}
We think of this as doing Bernoulli percolation on $F_1$ with parameter $p/q$ and then a random-cluster percolation with free boundary conditions on $F_2.$ For $\sigma \in F_1$ and a configuration $\xi$ on $\mathcal{T}_1,$ we clearly have 
\begin{align*}
&\mu_{\mathcal{T}_1}'\paren{\omega\paren{\sigma} = 1\mid \omega\paren{\tau} = \xi\paren{\tau} \text{ for all } \tau \in \mathcal{T}_1 \setminus \sigma} \\
&\qquad \leq \mu_{\mathcal{T}_1}\paren{\omega\paren{\sigma} = 1\mid \omega\paren{\tau} = \xi\paren{\tau} \text{ for all } \tau \in \mathcal{T}_1 \setminus \sigma}\,.
\end{align*}
By Lemma~\ref{lemma:subcomplex}, for $\sigma \in F_2$ we also have
\begin{align*}
&\mu_{\mathcal{T}_1}'\paren{\omega\paren{\sigma} = 1\mid \omega\paren{\tau} = \xi\paren{\tau} \text{ for all } \tau \in \mathcal{T}_1 \setminus \sigma} \\
&\qquad \leq \mu_{\mathcal{T}_1}\paren{\omega\paren{\sigma} = 1\mid \omega\paren{\tau} = \xi\paren{\tau} \text{ for all } \tau \in \mathcal{T}_1 \setminus \sigma}\,.
\end{align*}
The again applying Theorem~\ref{thm:stochdom}, we have \[\mu_{\mathcal{T}_1} \geq_{\mathrm{st}} \mu_{\mathcal{T}_1}'\,.\]

We now perform a splitting of the state of a plaquette into several Bernoulli variables similar to one found in Proposition 24 of~\cite{duncan2020homological}. We adapt some of the definitions used there. Let $S$ be the set of $i$-faces of $\mathcal{T}_1$ that intersect, but are not contained in $\mathcal{T}_2.$ For an $i$-face $v$ of $\mathcal{T}_2,$ let $J(v)$ be the set of all perpendicular $i$-faces that meet $v$ at an $(i-1)$ face. Then $v,$ $v+\vec{e}_1,$ and $J(v)$ are the $i$-faces of an $(i+1)$-face $w(v).$ Also, for a perpendicular $i$-face $u$ of $S,$ let $K(u)=\set{v:u\in J(v)}.$ Let $p_S$ satisfy 
\[ \frac{p}{q} = 1- \paren{1-p_S}^{2\paren{d-i}}\,.\]
For all pairs $\paren{v,u}$ where $v \in \mathcal{T}_2$ and $u\in J(v),$ let $\kappa\paren{v,u}$ be independent $\mathrm{Ber}\paren{p_S}$ random variables. Then by construction, Bernoulli percolation with parameter $p/q$ on $S$ is equivalent to setting each cell $v \in S$ to be open if and only if $\sum_{u \in J(v)} \kappa\paren{v,u} > 0.$ Given Bernoulli $p/q$ percolation on $F_1,$ let 
\[L = \set{v \in \mathcal{T}_2 : \kappa\paren{v,u} = 1 \text{ for each } u \in J\paren{v} \text{ and $v+\vec{e}_1$ is open}}\,.\]

Let the $\mu_{\mathcal{T}_1}''$ have as open plaquettes the union of the open plaquettes $\mu_{\mathcal{T}_1}' |_{\mathcal{T}_2}$ and the plaquettes in $L$(so all plaquettes of $F_1$ are all closed). Since the boundary of a plaquette in $L$ is necessarily the boundary of an open set of plaquettes, 
\[\mu_{\mathcal{T}_1}' \geq_{G} \mu_{\mathcal{T}_1}''.\]
Now let $p' = p+(1-p)\frac{p}{q}p_S^{2i}$ (i.e. the probability that a plaquette of $\mathcal{T}_2$ is either open or in $L$) and take $\mu_{\mathcal{T}_2,p'}$ to be the random-cluster model on $\mathcal{T}_2$ with parameter $p'$ instead of $p.$  Then we compare the probabilities in $\mu_{\mathcal{T}_1}''$ and in $\mu_{\mathcal{T}_2,p'}$ that a plaquette is open in the cases that its state does or does not affect $\mathbf{b}_{i-1}.$ In former case, the conditional probabilities are $\frac{p}{q} +(1-p)\frac{p}{q}p_S^{2i}$ and $\frac{p'}{q}$ respectively, and in the latter both are $p'.$ Thus, for any configuration $\xi$ on $\mathcal{T}_1$ and any $\sigma \in F_2,$ 
\begin{align*}
&\mu_{\mathcal{T}_1}''\paren{\omega\paren{\sigma} = 1\mid \omega\paren{\tau} = \xi\paren{\tau} \text{ for all } \tau \in \mathcal{T}_1 \setminus \sigma} \\
&\quad\geq \mu_{\mathcal{T}_1,p'}\paren{\omega\paren{\sigma} = 1 \mid \omega\paren{\tau} = \xi\paren{\tau} \text{ for all } \tau \in \mathcal{T}_1 \setminus \sigma}.
\end{align*}

Therefore we have 
\[\mu_{\mathcal{T}_1}'' \geq_{\mathrm{st}} \mu_{\mathcal{T}_2,p'}\]
by Theorem~\ref{thm:stochdom}, and so 
\[\mu_{\mathcal{T}_0} \geq_{G} \mu_{\mathcal{T}_2,p'}\,,\]
Now, as in~\cite{duncan2020homological}, we may take $p$ so that $p < p_u\paren{q,i,d-1} < p'.$ Then for each $N$ we have 
\[\mu_{\T^d_N,p,q,i}\paren{A} \geq \mu_{\T^{d-1}_N,p',q,i}\paren{A}\,,\]
so
\[\liminf \mu_{\T^d_N,p,q,i}\paren{A} \geq \liminf \mu_{\T^{d-1}_N,p',q,i}\paren{A} \geq 1/2\,,\]
and thus 
\[p_u\paren{q,i,d} \leq p < p_u\paren{q,i,d-1}\,.\]

Combining this with Proposition~\ref{prop:duality} also gives $p_u\paren{q,i,d-1} < p_u\paren{q,i+1,d}.$
\end{proof}

\appendix

\section{General Boundary Conditions}\label{app:boundaryconditions}
In this section we generalize the notion of boundary conditions. Let $\Omega$ be the space of configurations of $i$-plaquettes in $\Z^d.$ For a subset of vertices $V \subset \Z^d,$ let $F^k_V$ be the set of $k$-plaquettes with all vertices contained in $V.$ Then given $\xi \in \Omega,$ let 
\[\Omega_{\Lambda_n}^{\xi} = \set{\omega \in \Omega : \omega\paren{\sigma} = \xi\paren{\sigma} \text{ for all } \sigma \in F^k_{\Z^d} \setminus F^k_{\Lambda_{n-1}}}\,.\]

Intuitively, the boundary condition should describe how the states of external plaquettes affect the random-cluster measure within $\Lambda_n.$ In the classical model, this is done by keeping track of which vertices of $\partial \Lambda_n$ are connected externally. In higher dimensions we will need slightly more information, but the idea is the same. Let $P_{\xi,V}$ be the complex consisting of the union of the $(i-1)$-skeleton of $\Z^d$ and the open plaquettes of $\xi$ contained in $F^i_{V}.$ Let $D_n^{i-1}$ be the $(i-1)$-skeleton of $\partial \Lambda_n.$ Then we construct a cubical complex $Q_{\omega,\xi}$ (not necessarily a subcomplex of $\Z^d$) by taking $P_{\omega,\Lambda_n}$ and attaching a cubical complex $A_{\xi}$ so that
\begin{itemize}
    \item $A_{\xi} \cap P_{\omega,\Lambda_n} \subset D_n^{i-1}.$
    \item The map $\varphi_A : H_{i-1}\paren{D_n^{i-1};\,\F} \to  H_{i-1}\paren{A_{\xi};\,\F}$ induced by the inclusion $D_n^{i-1} \hookrightarrow A_{\xi}$ is surjective.
    \item The kernel of $\varphi_A$ is the same as the kernel of the map $H_{i-1}\paren{D_n^{i-1};\,\F} \to  \varphi_P : H_{i-1}\paren{P_{\xi,\Z^d\setminus \Lambda_{n-1}};\,\F}$ induced by the inclusion $D_n^{i-1} \hookrightarrow P_{\xi,\Z^d\setminus \Lambda_{n-1}}.$
\end{itemize}

Such an $A_{\xi}$ can be constructed by taking $P_{\xi}$ and filling the $(i-1)$-cycles that are not homologous to cycles in $D_n^{i-1}.$

Now we can define the plaquette random-cluster model on $\Lambda_n$ with boundary condition $\xi$ as follows:
\[\mu^{\xi}_{\Lambda_n}\paren{\omega} = \begin{cases}
\frac{1}{Z^{\xi}_{\Lambda_n}}\brac{\prod_{\sigma \in F^i_{\Lambda_n}} p^{\omega\paren{\sigma}}\paren{1-p}^{1-\omega\paren{\sigma}}}q^{\mathbf{b}_{i-1}\paren{Q_{\omega,\xi};\,\F}} & \omega \in \Omega_{\Lambda_n}^{\xi}\\
0 & \text{otherwise.}
\end{cases}
\]

In particular, the free boundary measure $\mu_{\Lambda_n}^{\mathbf{f}}$ and the wired boundary measure $\mu_{\Lambda_n}^{\mathbf{w}}$ are obtained by taking $\xi$ to be the all closed and all open configurations respectively.

\begin{Lemma}\label{lemma:positiveassociation}
Let $p \in \brac{0,1}, q \geq 1,$ and $n \in \N.$ Then for every $\xi \in \Omega,$ $\mu^{\xi}_{\Lambda_n}$ is positively associated.
\end{Lemma}

\begin{proof}
The proof is analogous to the proof of Theorem 4.14 of ~\cite{grimmett2006random}. Consider the plaquette random-cluster model $\mu_{\Lambda_n\cup A_{\xi}}.$ This satisfies the FKG lattice condition and is thus strongly postively associated. Then since $\mu^{\xi}_{\Lambda_n}$ is $\mu_{\Lambda_n\cup A_{\xi}}$ conditioned on the plaquettes of $A_{\xi}$ being open, it follows that $\mu^{\xi}_{\Lambda_n}$ is positively associated.
\end{proof}

\section*{Acknowledgments}
We would like to thank Sky Cao, Matthew Kahle, and David Sivakoff for interesting and useful discussions.

\bibliographystyle{alpha}
\bibliography{bibliography}

\newcommand{\etalchar}[1]{$^{#1}$}
\begin{thebibliography}{WHK{\etalchar{+}}07}

\bibitem[ACC{\etalchar{+}}83]{ACCFR83}
M.~Aizenman, J.~T. Chayes, L.~Chayes, J.~Fr\"{o}hlich, and L.~Russo.
\newblock On a sharp transition from area law to perimeter law in a system of
  random surfaces.
\newblock {\em Comm. Math. Phys.}, 92(1):19--69, 1983.

\bibitem[AF84]{aizenman1984topological}
Michael Aizenman and J{\"u}rg Fr{\"o}hlich.
\newblock Topological anomalies in the n dependence of the n-states {P}otts
  lattice gauge theory.
\newblock {\em Nuclear Physics B}, 235(1):1--18, 1984.

\bibitem[Alt78]{altes1978duality}
CP~Korthals Altes.
\newblock Duality for $z(n)$ gauge theories.
\newblock {\em Nuclear Physics B}, 142(3):315--326, 1978.

\bibitem[BDC12]{beffara2012self}
Vincent Beffara and Hugo Duminil-Copin.
\newblock The self-dual point of the two-dimensional random-cluster model is
  critical for $q\geq1$.
\newblock {\em Probability Theory and Related Fields}, 153(3-4):511--542, 2012.

\bibitem[Bre13]{bredon2013topology}
Glen~E Bredon.
\newblock {\em Topology and geometry}, volume 139.
\newblock Springer Science \& Business Media, 2013.

\bibitem[Cao20]{cao2020wilson}
Sky Cao.
\newblock Wilson loop expectations in lattice gauge theories with finite gauge
  groups.
\newblock {\em Communications in Mathematical Physics}, 380(3):1439--1505,
  2020.

\bibitem[CC84]{chayes1984correct}
JT~Chayes and L~Chayes.
\newblock The correct extension of the {F}ortuin-{K}asteleyn result to
  plaquette percolation.
\newblock {\em Nuclear Physics B}, 235(1):19--23, 1984.

\bibitem[Cha16]{chatterjee2016yang}
Sourav Chatterjee.
\newblock Yang--{M}ills for probabilists.
\newblock In {\em International Conference in Honor of the 75th Birthday of SRS
  Varadhan}, pages 1--16. Springer, 2016.

\bibitem[Cha20]{chatterjee2020wilson}
Sourav Chatterjee.
\newblock Wilson loops in {I}sing lattice gauge theory.
\newblock {\em Communications in Mathematical Physics}, 377(1):307--340, 2020.

\bibitem[Cha21]{chatterjee2021probabilistic}
Sourav Chatterjee.
\newblock A probabilistic mechanism for quark confinement.
\newblock {\em Communications in Mathematical Physics}, 385(2):1007--1039,
  2021.

\bibitem[DCRT19]{duminil2019sharp}
Hugo Duminil-Copin, Aran Raoufi, and Vincent Tassion.
\newblock Sharp phase transition for the random-cluster and {P}otts models via
  decision trees.
\newblock {\em Annals of Mathematics}, 189(1):75--99, 2019.

\bibitem[DKS20]{duncan2020homological}
Paul Duncan, Matthew Kahle, and Benjamin Schweinhart.
\newblock Homological percolation on a torus: plaquettes and permutohedra.
\newblock {\em arXiv preprint arXiv:2011.11903}, 2020.

\bibitem[DS]{PlaquetteSW}
Paul Duncan and Benjamin Schweinhart.
\newblock The plaquette {S}wendsen--{W}ang algorithm for {P}otts lattice gauge
  theory.
\newblock In preparation.

\bibitem[DW82]{druhl1982algebraic}
K~Dr{\"u}hl and H~Wagner.
\newblock Algebraic formulation of duality transformations for abelian lattice
  models.
\newblock {\em Annals of Physics}, 141(2):225--253, 1982.

\bibitem[ES88]{edwards1988generalization}
Robert~G Edwards and Alan~D Sokal.
\newblock Generalization of the {F}ortuin-{K}asteleyn-{S}wendsen-{W}ang
  representation and {M}onte {C}arlo algorithm.
\newblock {\em Physical review D}, 38(6):2009, 1988.

\bibitem[FK72]{fortuin1972random}
Cornelius~Marius Fortuin and Piet~W Kasteleyn.
\newblock On the random-cluster model: I. {I}ntroduction and relation to other
  models.
\newblock {\em Physica}, 57(4):536--564, 1972.

\bibitem[FLV21]{forsstrom2020wilson}
Malin~Pal{\"o} Forsstr{\"o}m, Jonatan Lenells, and Fredrik Viklund.
\newblock Wilson loops in finite abelian lattice gauge theories.
\newblock {\em to appear in Annales de l’Institut Henri Poincaré (B)
  Probabilités et Statistiques}, 2021.

\bibitem[Fr{\"o}79]{frohlich1979confinement}
Jurg Fr{\"o}hlich.
\newblock Confinement in $z_n$ lattice gauge theories implies confinement in
  ${S}{U}(n)$ lattice {H}iggs theories.
\newblock {\em Physics Letters B}, 83(2):195--198, 1979.

\bibitem[FS82]{frohlich1982massless}
J{\"u}rg Fr{\"o}hlich and Thomas Spencer.
\newblock Massless phases and symmetry restoration in abelian gauge theories
  and spin systems.
\newblock {\em Communications in Mathematical Physics}, 83(3):411--454, 1982.

\bibitem[GG06]{graham2006influence}
Benjamin~T Graham and Geoffrey~R Grimmett.
\newblock Influence and sharp-threshold theorems for monotonic measures.
\newblock {\em The Annals of Probability}, 34(5):1726--1745, 2006.

\bibitem[GGZ80]{ginsparg1980large}
Paul Ginsparg, Yadin~Y Goldschmidt, and Jean-Bernard Zuber.
\newblock Large $q$ expansions for $q$-state gauge-matter {P}otts models in
  {L}agrangian form.
\newblock {\em Nuclear Physics B}, 170(3):409--432, 1980.

\bibitem[GM90]{grimmett1990supercritical}
Geoffrey~Richard Grimmett and John~M Marstrand.
\newblock The supercritical phase of percolation is well behaved.
\newblock {\em Proceedings of the Royal Society of London. Series A:
  Mathematical and Physical Sciences}, 430(1879):439--457, 1990.

\bibitem[Gri04]{grimmett2004random}
Geoffrey Grimmett.
\newblock The random-cluster model.
\newblock In {\em Probability on discrete structures}, pages 73--123. Springer,
  2004.

\bibitem[Gri06]{grimmett2006random}
Geoffrey~R Grimmett.
\newblock {\em The random-cluster model}, volume 333.
\newblock Springer Science \& Business Media, 2006.

\bibitem[Gut80]{guth1980existence}
Alan~H Guth.
\newblock Existence proof of a nonconfining phase in four-dimensional u (1)
  lattice gauge theory.
\newblock {\em Physical Review D}, 21(8):2291, 1980.

\bibitem[Hat02]{hatcher2002algebraic}
Allen Hatcher.
\newblock {\em Algebraic Topology}.
\newblock Cambridge University Press, 2002.

\bibitem[HS16]{hiraoka2016tutte}
Yasuaki Hiraoka and Tomoyuki Shirai.
\newblock Tutte polynomials and random-cluster models in {B}ernoulli cell
  complexes (stochastic analysis on large scale interacting systems).
\newblock {\em RIMS Kokyuroku Bessatsu}, 59:289--304, 2016.

\bibitem[KMM04]{kaczynski2004computational}
Tomasz Kaczynski, Konstantin~Michael Mischaikow, and Marian Mrozek.
\newblock {\em Computational homology}.
\newblock Springer, 2004.

\bibitem[KPSS80]{kogut1980z}
JB~Kogut, RB~Pearson, J~Shigemitsu, and DK~Sinclair.
\newblock $z_n$ and $n$-state potts lattice gauge theories: Phase diagrams,
  first-order transitions, $\beta$ functions, and $1/n$ expansions.
\newblock {\em Physical Review D}, 22(10):2447, 1980.

\bibitem[KS82]{kotecky1982first}
Roman Koteck{\`y} and SB~Shlosman.
\newblock First-order phase transitions in large entropy lattice models.
\newblock {\em Communications in Mathematical Physics}, 83(4):493--515, 1982.

\bibitem[LMR89]{laanait1989discontinuity}
L~Laanait, A~Messager, and J~Ruiz.
\newblock Discontinuity of the {W}ilson string tension in the 4-dimensional
  lattice pure gauge {P}otts model.
\newblock {\em Communications in mathematical physics}, 126(1):103--131, 1989.

\bibitem[Men86]{menshikov1986coincidence}
Mikhail~V Menshikov.
\newblock Coincidence of critical points in percolation problems.
\newblock In {\em Soviet Mathematics Doklady}, volume~33, pages 856--859, 1986.

\bibitem[MMS79]{marra1979statistical}
R~Marra and S~Miracle-Sole.
\newblock On the statistical mechanics of the gauge invariant {I}sing model.
\newblock {\em Communications in Mathematical Physics}, 67(3):233--240, 1979.

\bibitem[MO82]{maritan1982gauge}
A~Maritan and C~Omero.
\newblock On the gauge {P}otts model and the plaquette percolation problem.
\newblock {\em Nuclear Physics B}, 210(4):553--566, 1982.

\bibitem[MP79]{mack1979comparison}
G~Mack and VB~Petkova.
\newblock Comparison of lattice gauge theories with gauge groups $z_2$ and
  $su(2)$.
\newblock {\em Annals of Physics}, 123(2):442--467, 1979.

\bibitem[OS78]{osterwalder1978gauge}
Konrad Osterwalder and Erhard Seiler.
\newblock Gauge field theories on a lattice.
\newblock {\em Annals of Physics}, 110(2):440--471, 1978.

\bibitem[Pis96]{pisztora1996surface}
Agoston Pisztora.
\newblock Surface order large deviations for {I}sing, {P}otts and percolation
  models.
\newblock {\em Probability Theory and Related Fields}, 104(4):427--466, 1996.

\bibitem[Sav16]{saveliev2016topology}
Peter Saveliev.
\newblock {\em Topology illustrated}.
\newblock Peter Saveliev, 2016.

\bibitem[Sei82]{seiler1982gauge}
Erhard Seiler.
\newblock {\em Gauge theories as a problem of constructive quantum field theory
  and statistical mechanics}.
\newblock Springer Berlin Heidelberg, 1982.

\bibitem[SW87]{swendsen1987nonuniversal}
Robert~H Swendsen and Jian-Sheng Wang.
\newblock Nonuniversal critical dynamics in monte carlo simulations.
\newblock {\em Physical review letters}, 58(2):86, 1987.

\bibitem[Weg71]{wegner1971duality}
Franz~J Wegner.
\newblock Duality in generalized {I}sing models and phase transitions without
  local order parameters.
\newblock {\em Journal of Mathematical Physics}, 12(10):2259--2272, 1971.

\bibitem[Weg17]{wegner2017duality}
Franz~J Wegner.
\newblock Duality in generalized {I}sing models.
\newblock {\em Topological Aspects of Condensed Matter Physics: Lecture Notes
  of the Les Houches Summer School: Volume 103, August 2014}, 103:219, 2017.

\bibitem[WHK{\etalchar{+}}07]{wipf2007generalized}
Andreas Wipf, Thomas Heinzl, Tobias Kaestner, Christian Wozar, et~al.
\newblock Generalized {P}otts-models and their relevance for gauge theories.
\newblock {\em SIGMA. Symmetry, Integrability and Geometry: Methods and
  Applications}, 3:006, 2007.

\bibitem[Wil74]{wilson1974confinement}
Kenneth~G Wilson.
\newblock Confinement of quarks.
\newblock {\em Physical review D}, 10(8):2445, 1974.

\bibitem[Wil04]{wilson2004origins}
Kenneth~G Wilson.
\newblock The origins of lattice gauge theory.
\newblock {\em arXiv preprint hep-lat/0412043}, 2004.

\bibitem[Yon78]{yoneya1978z}
Tamiaki Yoneya.
\newblock $z(n)$ topological excitations in {Y}ang-{M}ills theories: duality
  and confinement.
\newblock {\em Nuclear Physics B}, 144(1):195--218, 1978.

\end{thebibliography}

\end{document}